\documentclass{aimsmod} % 
\usepackage{amsmath}
\usepackage{paralist}
\usepackage[misc]{ifsym}
\usepackage{epsfig} %For pictures: screened artwork should be set up with an 85 or 100 line screen
\usepackage{epstopdf} %This is to transfer .eps figure to .pdf figure; please compile your article using PDFLaTex or PDFTeXify.
\usepackage[colorlinks=true]{hyperref}
\hypersetup{urlcolor=blue, citecolor=red}
\allowdisplaybreaks

     \def\DOI{10.48550/arXiv.2205.05795}

\usepackage{amssymb}
\usepackage{amsfonts}
\usepackage{mathtools}
\usepackage{xcolor}
\usepackage{subcaption}
\usepackage{tikz}
\usepackage{tikz-cd}
\usetikzlibrary{arrows,decorations.pathmorphing,backgrounds,positioning,fit,petri,calc,shapes.misc,decorations.markings}
\usepackage{float}
\usepackage{enumerate}
\usepackage{stackengine}
\setstackEOL{\\}
\usepackage[ruled,vlined]{algorithm2e}

\newtheorem{theorem}{Theorem}[section]

\newtheorem{proposition}[theorem]{Proposition}

\theoremstyle{definition}

\newtheorem{remark}[theorem]{Remark}
\newtheorem{example}[theorem]{Exampe}

\DeclareMathOperator*{\argmax}{arg\,max}
\DeclareMathOperator*{\argmin}{arg\,min}
\newcommand{\sphere}{\mathbb{S}}
\newcommand{\trans}{\operatorname{t}}
\newcommand{\realrad}[2]{\sqrt[\textup{re}]{{#2}_{#1}}}

\newcommand{\C}{\mathbb{C}}
\newcommand{\N}{\mathbb{N}}
\newcommand{\F}{\mathbb{F}}
\newcommand{\K}{\mathbb{K}}
\newcommand{\Q}{\mathbb{Q}}
\newcommand{\R}{\mathbb{R}}

\numberwithin{figure}{section}
\numberwithin{equation}{section}

\title[Algebraic Machine Learning]
{Algebraic Machine Learning with an Application to Chemistry} % Only the first word and proper nouns should be capitalized.

\author[Ezzeddine El Sai and Parker Gara and Markus J. Pflaum]{}


\thanks{$^*$Corresponding author: Markus J. Pflaum}

\begin{document}
\maketitle

% Enter the first author's name and email address; email addresses are required for each author.
% Use footnote notations to indicate address and affiliations with commas between numbers if more than one address applies; see below for examples.
\centerline{\scshape
Ezzeddine El Sai$^{{\href{ezzeddine.elsai@colorado.edu}{\textrm{\Letter}}}1}$
and Parker Gara$^{{\href{parker.gara@colorado.edu}{\textrm{\Letter}}}1}$
and Markus J. Pflaum$^{{\href{markus.pflaum@colorado.edu}{\textrm{\Letter}}}*1}$
}

\medskip

{\footnotesize
% Enter the full affiliation and country name:
% Do not insert commas or periods at the end of lines.
 \centerline{$^1$Department of Mathematics, University of Colorado, Boulder CO 80309, USA}
} % Do not forget to end {\footnotesize with the sign }

\medskip

%{\footnotesize
 % Enter the full affiliation and country name:
% \centerline{$^2$Affiliation, Country}
%}

\bigskip

% The name of the handling editor will be entered by AIMS production staff.
% "Communicated by Handling Editor" is not needed for special issue.
%mod \centerline{(Communicated by Handling Editor)}

%%%%%%%%%%%%%%%%%%%%%%%%%%%%%%%%%%%%%%%%%%%%%%%%%%%%%%%
%             5. ABSTRACT
%%%%%%%%%%%%%%%%%%%%%%%%%%%%%%%%%%%%%%%%%%%%%%%%%%%%%%%

 \begin{abstract}
As data used in scientific applications become more complex, studying the geometry and topology of data has become an increasingly prevalent part of data analysis. This can be seen for example with the growing interest in topological tools such as persistent homology. However, on the one hand, topological tools are inherently limited to providing only coarse information about the underlying space of the data. On the other hand, more geometric approaches rely predominantly on the manifold hypothesis which asserts that the underlying space is a smooth manifold. This assumption fails for many physical models where the underlying space contains singularities.

In this paper, we develop a machine learning pipeline that captures fine-grain geometric information without having to rely on any smoothness assumptions. Our approach involves working within the scope of algebraic geometry and algebraic varieties instead of differential geometry and smooth manifolds. In the setting of the variety hypothesis, the learning problem becomes to find the underlying variety using sample data. We cast this learning problem into a Maximum A Posteriori optimization problem which we solve in terms of an eigenvalue computation. Having found the underlying variety, we explore the use of Gr\"obner bases and numerical methods to reveal information about its geometry. In particular, we propose a heuristic for numerically detecting points lying near the singular locus of the underlying variety.   
%This is the abstract of your article. It should not exceed \textbf{200} words and needs to be concise and factual. State the purpose of the research, the principal results, and conclusion.
\end{abstract}

%%%%%%%%%%%%%%%%%%%%%%%%%%%%%%%%%%%%%%%%%%%%%%%%%%%%%%
%                   6. BODY
%%%%%%%%%%%%%%%%%%%%%%%%%%%%%%%%%%%%%%%%%%%%%%%%%%%%%%

% Only the first word and proper nouns of section titles should be capitalized.
% The title of section 1:
%\section{Introduction}
\section{Introduction}
Studying the geometry of data can provide insight into high dimensional, highly complex datasets. Applications include dimensionality reduction \cite{mcinnes2020umap}, computer vision \cite{10.5555/1145132}, chemistry \cite{Martin2010} and medicine\cite{AhmadiMoughari2020ADRMLAD}. Geometric methods in machine learning rely predominantly on \textit{the manifold hypothesis} \cite{Fefferman2016} which asserts that sample data $\Omega\subset \mathbb{R}^{n}$ in fact lie on a smooth submanifold $M\subset \mathbb{R}^{n}$ of dimension $<<n$. Algorithms based upon the this hypothesis are often called \textit{manifold learning} methods. Examples include PCA and nonlinear PCA \cite{Leeuw2013HistoryON}, Isomap \cite{Tenenbaum2000AGG}, and UMAP \cite{mcinnes2020umap}. However, the manifold hypothesis does no always hold, especially when the underlying geometry of the data contains singularities. Singularities are in fact ubiquitous in mathematics and appear often in physical models \cite{doi:10.1021/jp962746i,reiman1996singular}. This explains  the need to go beyond the manifold hypothesis.

To transition from the smooth to the singular setting we replace the language of differential geometry and smooth manifolds with algebraic geometry and algebraic varieties. At a basic level, algebraic geometry studies the geometry of the zeros of systems of polynomials. Those zero sets are called \textit{algebraic varieties}. At the heart of algebraic geometry is the duality between geometry and algebra, which allows us to jump back and forth between geometric spaces and computable algebraic procedures. Furthermore, algebraic geometry offers a natural setting for studying and working with singularities. As we shall see, this \textit{variety hypothesis} provides a great amount of flexibility and it lends itself well for computations. 

To examine data in the singular setting we introduce the \textit{algebraic machine learning pipeline} depicted in Figure \ref{fig:pipeline}. This pipeline combines ideas from algebraic geometry and machine learning and it is the main subject of this paper.

\medskip

\begin{figure}
\centering
\begin{tikzcd}
& & \fbox{\Centerstack{Algebraic\\Computations}} \arrow[rd] & & \\
\fbox{\Centerstack{Data}} \arrow[r] & \fbox{\Centerstack{Learning the\\Underlying Variety}} \arrow[ru] \arrow[rd] & & \fbox{\Centerstack{Geometric\\Insights}} \\
& & \fbox{\Centerstack{Numerical\\Computations}} \arrow[ru] & &
\end{tikzcd}
\caption{}\label{fig:pipeline}
\end{figure}

\medskip
\noindent
In Section \ref{sec:real-algebraic-geometry}, we give a brief overview of real algebraic geometry, building the language we will be using throughout the paper. In Section \ref{sec:learning-variety}, we discuss the first stage of learning the underlying variety. We introduce the following learning problem:
\begin{itemize}
\item[(LP)] 
  \textit{Given a dataset $\Omega=(a_{1},\ldots ,a_{m})$ of points in $[0,1]^{n}$ sampled from a variety $V \subset \mathbb{R}^n$,
  find that variety.}   
  %\textit{Given a dataset $\Omega=(a_{1},\ldots ,a_{m})$  of points in $[0,1]^{n}$ find the variety $V \subset \mathbb{R}^n$ such that the points in $\Omega$ belong to $V$.}
\end{itemize}
We interpret this learning problem as a Maximum A Posteriori (MAP) problem which can be solved in terms of an eigenvalue computation. This gives a more rigorous justification for the algorithm presented in \cite{Breiding2018}. Building upon this work, we present a closely related algorithm as Algorithm \ref{MAP-Model} as well as a possible enhancement as Algorithm \ref{Intersected-MAP-Model}.
% \color{red}
Let us mention that restricting the sample points to be contained in a unit
hypercube might not be an obvious choice, but restriction to some bounded
region is necessary for the later statistical arguments to make sense, see Remark \ref{rem:lojasiewicz}. 
Indeed, the paper \cite{8999343} by Dufresne et al.\ makes similar assumptions. 
%\color{black}

In Section \ref{sec:algebraic-computations}, we introduce the algebraic computations stage.
Using the learned variety $V$ from the previous stage, we explore the use of Gr\"obner basis computations to reveal information about $V$. This includes invariants like the dimension and the
irreducible decomposition.

In Section \ref{sec:singular-heuristics}, we consider the numerical computations stage which involves
working directly with the points in $\Omega$ and additional samples taken from $V\cap[0,1]^{n}$. In particular, we introduce a method called the \textit{singularity heuristic}
for detecting points in $[0,1]^{n}$ lying near the singular locus of $V$.

In Section \ref{sec:results}, we test the algebraic machine learning pipeline on synthetic data and on chemical data sampled from the conformation space of cyclooctane \cite{Martin2010}. Finally, in Section \ref{sec:related-work}, we discuss other work related to the algebraic machine learning pipeline.

%\textbf{Acknowledgement:}
%The authors greatfully acknowledge support by the NSF under award OAC 1934725 for the project
%\emph{DELTA: Descriptors of Energy Landscape by Topological Analysis}. M.J.\ Pflaum thanks David Jonas and David Walba for crucial explanations  on the conformations of cylooctane. The authors also thank Henry Adams, Howy Jordan and Trevor Sesnic for helpful discussions. 

%Use this AIMS template to prepare your tex file after your article is accepted by an AIMS journal. Read all information including that which is proceeded by a \% sign. These are important instructions and explanations. Thank you for your cooperation.

% The title of section 2:
% \section{Examples}
\section{Background in Real Algebraic Geometry}
\label{sec:real-algebraic-geometry}
For the reader's convenience, we collect in this section several fundamental concepts from real and complex
algebraic geometry which are used throughout this paper; see \cite{BocCosRoyRAG} and \cite{HarAG} for further details.
In this paper, we will be working mostly over the field $\R$ and  its subfields $\Q$ and
$\R_{\textup{alg}} := \overline{\Q}\cap \R$ of rational and  real algebraic numbers, respectively. The symbol $\K$
always denotes a field which we assume to be of characteristic $0$ unless stated otherwise.
\subsection{Ordered Fields}

A total order $\leq$ on a field $\K$ is called a
\textit{field order} if it is compatible with the field operations in the
following sense:
\begin{enumerate}[(M1)]
\item\label{ite:monotonyaddition} If $a\leq b$, then for all $c\in\K$ the
    relation  $a+c\leq b+c$ holds true.
\item\label{ite:monotonymultiplication} If $0\leq a$ and $0\leq b$, then $0\leq a\cdot b$.
\end{enumerate}
A field $\K$ endowed with a field order $\leq$ is an \textit{ordered field}.
We always assume $\R$ and $\Q$ to be equipped with the standard field order
$\leq$. The set of natural numbers $\N \subset \Q$ carries the induced
natural order.
Note that the field of complex numbers $\mathbb{C}$ and finite fields
cannot be endowed with the structure of an ordered field. For finite fields this follows from the
simple observation that for any ordered field $\K$ the natural map $\N \to \K$ 
is strictly increasing, hence $\K$ needs to be infinite. For the field of complex numbers see
Example \ref{ex:cplxnumbers} (a) below. 

The order on an ordered field $(\K,\leq)$ is completely characterized by the
\emph{positive  cone} it generates, which is the set
$C = \{ a \in \K \mid 0 \leq a \} $. Let us explain this in some more detail. Recall from
e.g.~\cite[Sec.~1.1]{BocCosRoyRAG} that by a \emph{proper cone} of $\K$ one understands a
subset $C \subset \K$ which fulfills the relations
\begin{equation}
  \label{eq:cone-axioms}
  C + C \subset C , \quad C \cdot C \subset C , \quad \K^2 \subset C, \text{ and } -1 \notin C \ ,
\end{equation}
where $ C+C = \{a+b| a, b \in C \}$, $C\cdot C = \{ a \cdot b|a, b \in C\}$ and
$\mathbb{K}^2 = \{a\cdot a| a \in K\}$. 
A proper cone $C$ of $\K$ is a \emph{positive cone} if in addition 
\begin{equation}
  \label{eq:positivity-axiom}
  \K = C \cup (-C) \ .
\end{equation}
A positive cone $C$ of $\K$ defines a unique field order $\leq_C$
% such that $C = \{ a \in \K \mid 0 \leq_C a \} $
by putting $a \leq_C b$ if and only if $b-a \in C$.
The thus  defined relation $\leq_C$ satisfies property (M\ref{ite:monotonyaddition}) by definition
and (M\ref{ite:monotonymultiplication}) since $C \cdot C \subset C$.
Finally, the cone $C$ is given by the set $\{ a \in \K \mid 0 \leq_C a \} $ which concludes 
the argument that field orders on $\K$ are in bijective correspondence with positive cones $C \subset \K$.
\begin{example}\label{ex:cplxnumbers}
  \begin{enumerate}[(a)]
  \item 
  The field of complex numbers $\C$ does not possess a positive cone since if $C \subset \C$ were such a cone,
  then $-1 = i^2 \in C$ which contradicts \eqref{eq:cone-axioms}.
\item
  For the fields $\K =\R$ or $\K = \Q$, the standard field order coincides with the order
$\leq_P$ associated to the positive cone
\[ P = \big\{ a \in \K \, \big| \:  \exists\,  a_1,\ldots, a_n \in \K : \: a = \sum_{i=1}^n a_i^2 \, \big\} \ . \]
\end{enumerate}
\end{example}

An ordered field $(\K,\leq)$ is called \textit{real closed} if every polynomial $f\in\K[x]$ of odd degree has a root in $\K$ 
and if for every positive element $a\in\K$ there exists $b\in\K$ such that $a=b^{2}$. For example, $(\mathbb{R},\leq)$ is real closed, but
$(\mathbb{Q},\leq)$ is not since the element $2$ is positive but does not have a rational square root. 
For a real closed field $(\K,\leq)$, the vector space $\K^n$ can be endowed with the $\K$-valued metric
\begin{equation}\label{eq:metric-real-closed-field}
 \|a-b\|_{2}=\sqrt{(a_{1}-b_{1})^{2}+ \ldots +(a_{n}-b_{n})^{2}}\ , \quad a,b \in \K^n \ .
\end{equation}
If $\K\subset \R$ is a real closed subfield, this results in the standard Euclidean metric restricted to $\K^n$.

The \textit{real closure} of an ordered field $(\K,\leq_\K)$ is an ordered field $(\F,\leq_{\F})$ which is real closed and
extends $(\K,\leq_\K)$, that is, $\K\subset \F$ and $C_\K\subset C_\F$, where  $C_\K$ and $C_\F$ are the positive cones
in $\K$ and $\F$, respectively. Real closures exist and are essentially unique. For example, the real closure of $(\mathbb{Q},\leq)$ is
$(\mathbb{R}_{alg},\leq_P)$, where  $P$ is understood as above to be the set of sums of squares in $\mathbb{R}_{alg}$.

%As long as the relevant field is clear from context, we will suppress the subscripts on the sets $re$.

%A field $\K$ is \textit   {formally real} if it can be endowed with a field ordering. An ordering on $\K$ can be represent by a subset $\tau\subset \K$ satisfying:
%\begin{itemize}
%    \item $\tau+\tau\subset \tau $
%    \item $\tau\cdot\tau\subset \tau $
%    \item $\K^{2}\subset \tau$ where $\K^{2}$ is the set of squares $\{a^{2}|a\in\K\}$
%    \item $-1\not\in \K$
%   \item $\tau\cup -1\cdot \tau=\K$
%\end{itemize}
%We say $a\leq b$ if and only if $b-a \in \tau$. As an example $\tau$

\subsection{The Nullstellensatz}
The set of polynomials in $n$ variables over a field $\K$ forms a ring
$\K[x_{1},\ldots ,x_{n}]$. Given a subset $S$ of the polynomial ring $\K[x_{1},\ldots ,x_{n}]$,
we denote by $Z_\K (S)$ or more briefly by $Z(S)$ the \emph{zero-set} of $S$, that is the set 
\[ 
   Z(S) = Z_\K(S) = \{a\in \K^{n}| f(a)=0 \text{ for all } f \in S\} \ .
\]
By $\langle S \rangle_\K$ or more shortly by $\langle S \rangle$ when the ground field is clear,
one denotes the ideal generated by $S$ which is the intersection of all ideals in
$\K[x_{1},\ldots ,x_{n}]$ containing $S$. Both the set $S$ and the ideal $\langle S \rangle$ have the same
zero
set $Z(S)=Z(\langle S\rangle)$ since if $f_{1}(a)= \ldots =f_{k}(a)=0$ for polynomials $f_{1},\ldots ,f_{k}$
and $a\in \K^n$, then
$g_{1}(a)\cdot f_{1}(a)+ \ldots +g_{k}(a)\cdot f_{k}(a)=0$ for all $g_{1},\ldots ,g_{k} \in \K[x_{1},\ldots ,x_{n}]$.

Since $\K[x_{1},\ldots ,x_{n}]$ is noetherian, any ideal $I\subset \K[x_{1},\ldots ,x_{n}]$ can be written as
$\langle f_{1},\ldots ,f_{k}\rangle$ for some $k$ and $f_{1},\ldots ,f_{k}\in I$. Therefore, starting with a possibly
infinite set $S$, $Z(S)$ can always be rewritten as $Z(S')$ for some finite set of polynomials $S'$.

Over a subfield of $\mathbb{R}$, this result can be sharpened further. For notional convenience, we will denote $Z(\{f\})$ by $Z(f)$. Note that the following proposition does not hold true over the field $\mathbb{C}$.

\begin{proposition}
\label{prop:real-variety-zero-set-single-polynomial}
  Let $\K$ be a subfield of $\mathbb{R}$. If $S$ is a non-empty subset of the polynomial ring $\K[x_{1},\ldots ,x_{n}]$, then
  $Z(S)$ can be written as $Z(f)$ for a single polynomial $f\in \K[x_{1},\ldots ,x_{n}]$.
\end{proposition}

\begin{proof}
  First rewrite $Z(S)$ as $Z(S')$ for some finite set $S'\subset \K[x_{1},\ldots ,x_{n}]$.
  Let $S'=\{f_{1},\ldots ,f_{k}\}$. If $f_{1}(a)= \ldots =f_{k}(a)=0$, then $f_{1}^{2}(a)+ \ldots +f_{k}^{2}(a)=0$ and 
  $Z(S')\subset Z(f^{2}_{1}+ \ldots +f^{2}_{k})$. Conversely,
  $f^{2}_{1}(a)+ \ldots +f^{2}_{k}(a)=0$ implies $f^{2}_{1}(a)= \ldots =f^{2}_{k}(a)=0$ which then entails
  $f_{1}(a)= \ldots =f_{k}(a)=0$. Hence $ Z(\{f_{1},\ldots , f_{k}\})=Z(f_{1}^{2}+ \ldots +f_{k}^{2})$.
\end{proof}

A set $V\subset \K^{n}$ is called an \emph{algebraic variety} or a \emph{variety} if $V=Z(S)$ for some set
$S\subset \K[x_{1},\ldots ,x_{n}]$. Depending on whether $\K=\R$ or $\K=\C$ the variety is called \emph{real}
or \emph{complex}. 

Similar to the zero-set map $Z$ one also has the \textit{vanishing ideal} map $J$ which maps a subset $A\subset \K^{n}$ to
the set of all polynomials vanishing on $A$ that is to the set 
\[ J(A)=J_\K(A)= \big\{ f\in\K[x_{1},\ldots ,x_{n}] \, \big| \: f(a)=0 \text{ for all } a\in A \big\} \ . \] 
To see why $J(A)$ is an ideal, note that if $f_{1}(a)=f_{2}(a)=0$ for $a\in A$,
then $f_{1}(a)\pm f_{2}(a)=0$ and $q(a)\cdot f_{1}(a)=0$ for any polynomial $q$.
%We sometimes write $J_\K (A)$ instead of $J (A)$ to denote that the underlying ground field is $\K$. 

Starting with an ideal $I\subset \K[x_{1},\ldots ,x_{n}]$, it is not necessarily the case that $I=J(Z(I))$. This is because there
might be other polynomials vanishing on $Z(I)$ which  were not included in $I$, so one can only guarantee that $I\subset J(Z(I))$.

\begin{example} Consider the ideal $\langle x^{2} \rangle\subset\mathbb{C}[x]$.
  Then $Z(\langle x^{2} \rangle)=\{0\}$. However, 
  \[ J(Z(\langle x^{2} \rangle))=J(\{0\})=\langle x\rangle \ , \]
  so $x$ is in $J(Z(\langle x^{2} \rangle))$ but not in $\langle x^{2} \rangle$.
\end{example}

\noindent
Over an algebraically closed field, such as $\mathbb{C}$, the missing polynomials are characterized by the \emph{radical} of $I$, defined as 
\[
  \sqrt{I} = \big\{ f \, \big| \: \exists r\in \N_{>0}  : \: %\text{ such that }
  f^{r}\in I \big\} \ .
\]
Note that $\sqrt{I}$ is itself an ideal; see e.g.\ \cite[X,\,\S 2]{Lang:Algebra}. 

\begin{theorem}[Hilbert's Nullstellensatz] Over an algebraically closed field $\K$ the equality 
  \[ J(Z(I))= \sqrt{I} \]
  holds for every ideal $I\subset \K[x_{1},\ldots ,x_{n}]$. 
\end{theorem}

An ideal $I$ which satisfies $I=\sqrt{I}$ is called a \emph{radical} ideal. Over algebraically closed fields, Hilbert's Nullstellensatz establishes a one-to-one correspondence between varieties and radical ideals in the polynomial ring. 
For ordered fields  there  is a similar but more subtle result. To formulate it we need some further notation. 
Let $(\K,\leq)$ be an ordered field.
The \emph{real radical} of an ideal $I \subset \K[x_{1},\ldots ,x_{n}]$ then is defined as
\[
  \realrad{}{I}=
  \big\{f \in \K[x_{1},\ldots ,x_{n}] \, \big| \: \exists 
  r \in \N_{>0}, \, f_{1},\ldots ,f_{k}\in \K[x_{1},\ldots ,x_{n}]  : \:
  % \text{ such that }
  f^{2r}+\sum_{i=1}^{k}f_{i}^{2}\in I\big\} \ .
\]
An ideal $I\subset \K[x_{1},\ldots ,x_{n}]$ satisfying $I=\realrad{}{I}$ is called \textit{real}.
If $(\F,\leq)$ denotes the real closure of the ordered field $(\K,\leq)$ and if $A$ is a subset of $\F^n$,
then the vanishing ideal of $A$ in  $\K[x_{1},\ldots ,x_{n}]$ is the ideal
\[ J_{\F} (A) \cap \K[x_{1},\ldots ,x_{n}] 
  % = \{ f\in\F[x_{1},\ldots ,x_{n}]|f(a)=0 \text{ for all } a\in A\} \cap \K[x_{1},\ldots ,x_{n}] 
\ . \]
We denote this ideal by $J_{\K}(A)$.  

\begin{theorem}[Real Nullstellensatz {\cite[Thm.~2.8]{Neuhaus1998}}] Let $(\K,\leq)$ be an ordered field and $(\F,\leq)$ be its
  real closure. If $I\subset \K[x_{1},\ldots ,x_{n}]$ is an ideal, then  
  $J_{\K}(Z_{\F}(I))=\realrad{}{I}$. 
  In particular, if $S\subset \mathbb{Q}[x_{1},\ldots ,x_{n}]$, then 
  $J_\mathbb{Q}(Z_{\mathbb{R}_{\textup{alg}}}(S))=\realrad{\Q}{\langle S\rangle}$.
\end{theorem}
Note that $Z_{\F}(I)$ is the variety in $\F^{n}$ obtained by viewing $I$ as a set of polynomials
in $\F[x_{1},\ldots ,x_{n}]$, and that $J_{\K}(Z_{\F}(I))$ is the ideal of $\K$-polynomials vanishing on
$Z_{\F}(I)$.  

%\color{red}
\begin{remark}
  We are ultimately interested in varieties over $\mathbb{R}$, but for computational purposes we restrict to varieties 
  in $\mathbb{R}^{n}_{\textup{alg}}$ defined by sets of polynomials $S\subset \mathbb{Q}[x_{1},\ldots ,x_{n}]$. 
  Therefore,  we will mostly work over ideals $I\leq \mathbb{Q}[x_{1},\ldots ,x_{n}]$.
  %To facilitate the comparison of vanishing ideals and zero sets over different ordered ground fields $\K$ and their real closures, 
  %we introduced the above notation $J_\K (A)$ and $Z_\K (S)$ for the vanishing ideal of a subset $A\subset \K^n$ and the
  %zero set of a subset $S\subset \K[x_1,\ldots,x_n]$, respectively. 
\end{remark}
%\color{black}

The following density result is concerned with the metric given by Eq.~\eqref{eq:metric-real-closed-field} and 
is an immediate consequence of \cite[Prop.\ 5.3.5]{BocCosRoyRAG}.

\begin{theorem}
  If $\K\subset \F$ are both real closed fields subfields of $\mathbb{R}$ and
  $S\subset \K[x_{1},\ldots ,x_{n}],$ then
  $Z_{\K}(S)$, viewed as a subset of $\F^{n}$, is dense in $Z_{\F}(S)$ with respect to the metric
  $\|\cdot \|_{2}$.
  In particular, if $S\subset \mathbb{Q}[x_{1},\ldots ,x_{n}]$, then $Z_{\mathbb{R}_{\textup{alg}}}(S)$ is dense in
  $Z_{\mathbb{R}}(S)$.
\end{theorem}

Intuitively, this means that we can approximate points in $Z_{\mathbb{R}}(S)$ arbitrarily closely by points in 
$Z_{\mathbb{R}_{\textup{alg}}}(S)$. If $f$ vanishes on $Z_{\mathbb{R}_{\textup{alg}}}(S)$ then since $f$ continuous 
on $\mathbb{R}^{n}$ and $Z_{\mathbb{R}_{\textup{alg}}}(S)$ is dense in $Z_{\mathbb{R}}(S)$, $f$ must also vanish on 
$Z_{\mathbb{R}}(S)$. Therefore, by the real Nullstellensatz 
\[ \realrad{\Q}{\langle S\rangle}=J_{\mathbb{Q}}(Z_{\mathbb{R}_{\textup{alg}}}(S))=J_{\mathbb{Q}}(Z_{\mathbb{R}}(S)) \ . \]

\subsection{Properties of Algebraic Varieties}
Real varieties inherit topological and differential structures based on their embedding in
$\mathbb{R}^{n}$ and form a very large class of spaces. For example, a famous result due to John Nash \cite{nash1952} says that
every closed manifold is diffeomorphic to a real algebraic variety. Moreover, by the fundamental work \cite{Whitney65}
of Hassler Whitney, every algebraic variety carries a unique stratification which is minimal among all stratifications fullfilling Whitney's condition (b).
% More precisely,
This means that every algebraic variety can be decomposed into
finitely many pairwise disjoint smooth manifolds called \emph{strata} such that two natural regularity conditions are
fulfilled, namely the condition of frontier and Whitney's condition (b). Moreover, there is a smallest among all
such decompositions which we call the \emph{canonical stratification} of an algebraic variety.
In this paper, we do not need the subtleties of the construction of
the canonical stratification of an algebraic variety, just its existence,
and therefore refer the interested reader to \cite[Chpt.\ 1]{PflAGSSS} for the
technical details on stratified spaces, the condition of frontier and the Whitney conditions.
Since manifolds of different dimensions may appear in the canonical stratification, varieties can also possess singularities
and show non-smooth behavior.
This makes the variety hypothesis, i.e.\ the claim that the underlying space of the data is an 
algebraic variety, a very reasonable assumption especially for scientific and computational purposes. 

\begin{example}
\label{ex:union-sphere-plane}  
  Consider the sphere of radius $\frac{1}{2}$ centered at $(\frac{1}{2},\frac{1}{2},\frac{1}{2})$ and the plane $x-y=0$. Their union is
  a variety which we can represent as the zero set of the polynomial
  \[
    f (x,y,z)= \left( (x-\frac{1}{2})^{2}+(y-\frac{1}{2})^{2}+(z-\frac{1}{2})^{2}-\frac{1}{4}\right)\cdot(x-y) \ .
  \]
  The variety $V= Z(f)$ is illustrated in Figure \ref{fig:zerolocus}.
  Its singular points are the points in the intersection of the sphere with the plane. Note that any point in the intersection
  must satisfy $y=x$, so one may set $y=x$ in the equation of the sphere and then observes that the singular locus of
  the variety $V$ is given by the circle $Z(\langle 2x^{2}-2x+z^{2}-z+\frac{1}{2}\rangle)\cap Z(\langle x-y\rangle)$
  illustrated in Figure \ref{fig:singularlocus}.
  
\begin{figure}[H]
 \centering
 \begin{subfigure}{.45\textwidth}\centering
 \includegraphics[scale=.3]{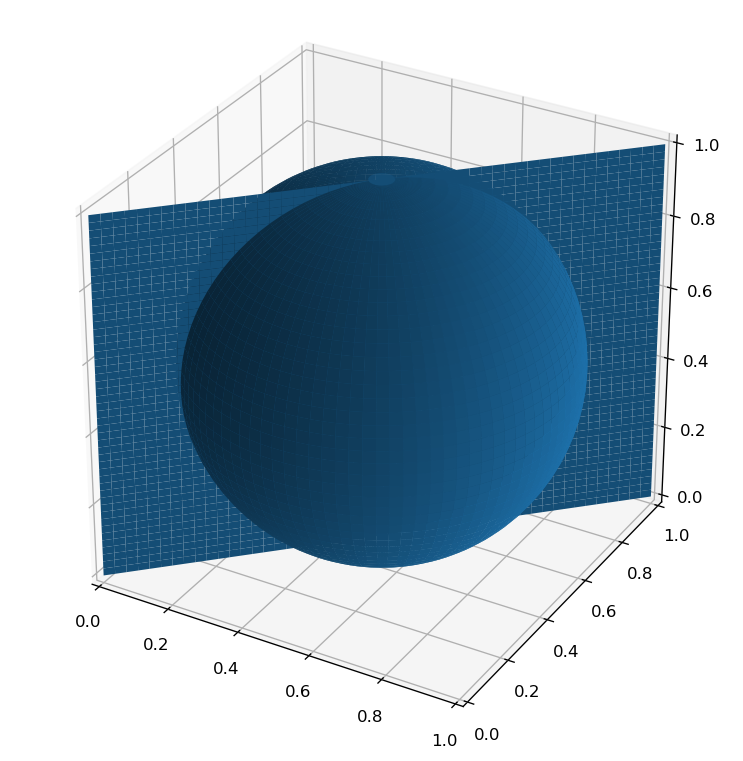}
 \caption{Zero locus of $f$}\label{fig:zerolocus}
 \end{subfigure}
 \begin{subfigure}{.45\textwidth}\centering
 \includegraphics[scale=.3]{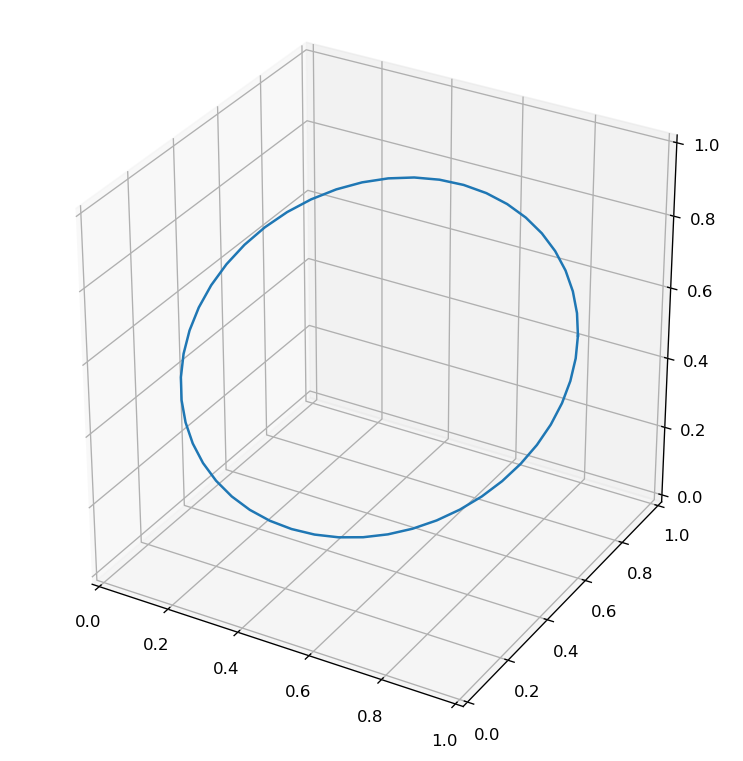}
 \caption{Singular locus}\label{fig:singularlocus}
\end{subfigure} \caption{The union of a sphere and a plane as a variety}\label{fig:varietyunionsphereplane}
\end{figure}
\end{example}

In what follows we review the concepts of dimension, the Jacobians, and the singular locus of a variety.
Let $V$ be a variety and let $J(V)\subset \K[x_{1},\ldots ,x_{n}]$ be its vanishing ideal.
% $J(V)=\langle f_{1},\ldots ,f_{k}\rangle$.
The \emph{coordinate ring} of $V$, denoted $\K(V)$, is the quotient ring $\K[x_{1},\ldots ,x_{n}]/J(V)$. It
contains all the information about the variety $V$. In fact, $V$ can be reconstructed from $\K(V)$
using its spectrum; see \cite{HarAG}. 
The \emph{Krull dimension} of a commutative ring $R$, denoted $\dim (R)$, is the length $n$ of the longest
chain $\langle 0\rangle= P_{0}\lneq P_{1}\lneq  \ldots \lneq P_{n}=\K(V)$ of prime ideals in $\K(V)$. 
%\color{red}
In case the variety is smooth, the Krull dimension of its coordinate ring coincides with its dimension 
as a real or complex manifold, respectively. 
%\color{black}
The following result gives a geometric interpretation of the Krull dimension of a coordinate ring.
It is an immediate consequence of \cite[Prop.~2.8.5]{BocCosRoyRAG}.

\begin{proposition}
  For every real algebraic variety $V$,
  $\dim (\mathbb{R}(V))=\dim(V)$, where
  $\dim (V)$ is the maximal dimension of the manifolds in the canonical Whitney
  stratification of $V$. 
\end{proposition}

Even though a variety $V$ is completely characterized by the ideal $J_{\mathbb{R}}(Z_{\mathbb{R}}(S))$,
the subset $\realrad{}{\langle S\rangle}$
still carries useful information.

\begin{proposition}\label{prop:dimension-geometric-algebraic}
  If $S\subset \mathbb{Q}[x_{1},\ldots ,x_{n}]$, then
  \[
    \dim (Z_{\mathbb{R}}(S))=\dim(\mathbb{R}_{\textup{alg}}(Z_{\mathbb{R}_{\textup{alg}}}(S)))=
    \dim (\mathbb{Q}[x_{1},\ldots ,x_{n}]/\realrad{\Q}{\langle S\rangle}) \ .
  \]  
\end{proposition}
\begin{proof}
  By \cite[Prop.~7]{Becker1993},
  $\dim (\mathbb{R}_{\textup{alg}}(Z_{\mathbb{R}_{\textup{alg}}}(S))) = \dim(\mathbb{Q}[x_{1},\ldots ,x_{n}]/\realrad{}{\langle S\rangle})$.
  Furthermore, the Krull dimension is preserved under field extensions, see e.g.\ \cite[II.~Ex.~3.20]{HarAG}.
  Hence $\dim(\mathbb{R}_{\textup{alg}}(Z_{\mathbb{R}_{\textup{alg}}}(S)))=\dim(\mathbb{R}(Z_{\mathbb{R}}(S)))$,
  which coincides with  $\dim(Z_{\mathbb{R}}(S))$.
\end{proof}

Related to the dimension are the singularities of a variety $V\subset\K^{n}$.
Express $J(V)=\langle f_{1},\ldots ,f_{k}\rangle$ with appropriate polynomials
$f_{1},\ldots, f_{k}\in \K [x_1,\ldots ,x_n]$. 
The \emph{Jacobian} of $V$ at the point $a\in V$ then is given by the matrix
\[
  \operatorname{Jac}_a (f_{1},\ldots ,f_{k})=\begin{pmatrix}
\partial_{x_{1}} f_{1}(a) &\dots&\partial_{x_{n}} f_{1}(a)\\
\vdots & \ddots & \\
\partial_{x_{1}} f_{k}(a) & &\partial_{x_{n}} f_{k}(a) 
\end{pmatrix} \ .
\]

\noindent
A \emph{singular point} or \emph{singularity} of $V$ now is a point $a\in V$ such that
\begin{equation*}
  \operatorname{rk} \big( \operatorname{Jac}_a (f_{1},\ldots ,f_{k}) \big) < n- \dim(V) \ .
\end{equation*}
Note that this condition does not depend on the choice of generators $f_{1},\ldots ,f_{k}$
for $J(V)$.
Conversely, a \emph{non-singular} point is a point $a\in V$ where
\[ \operatorname{rk} \big( \operatorname{Jac}_a(f_{1},\ldots ,f_{k})\big) = n-\dim(V) \ . \]
The set of all singular points of $V$ is called the \textit{singular locus} of $V$ and is denoted
$\operatorname{Sing} (V)$. One immediately checks that $\operatorname{Sing} (V)$ is 
again a variety by observing that $\operatorname{rk}\big(\operatorname{Jac}_a(f_{1},\ldots ,f_{k}\big))<n-r$
if and only if all $(n-r)$-minors of $\operatorname{Jac}_a(f_{1},\ldots ,f_{k})$ vanish.

For our purposes, it will be more convenient to use the following charactarization of
singularities which follows from \cite[Prop.~3.3.10]{BocCosRoyRAG}. 

\begin{theorem}\label{thm:singular-locus-hypersurface}
If $V\subset\mathbb{R}^{n}$ is a variety such that 
\begin{enumerate}[(i)]
    \item $\dim(V)=n-1$, and 
    \item $V=Z(f)$ for an irreducible polynomial $f$,
\end{enumerate}
then $a\in V$ is a singular point if and only if
$\big(\frac{\partial f}{\partial x_{1}}(a),\ldots ,\frac{\partial f}{\partial x_{n}}(a)\big)=
(0,\ldots ,0)$.
\end{theorem}

By subvariety of a variety $V$ one understands a subset $W\subset V$ which is itself a variety.
The subvariety is \emph{proper} if $W\subsetneq V$. 
A variety $V$ is called \emph{reducible} if it is the union of two or more proper non-empty subvarieties, 
otherwise it is called \emph{irreducible}. This suggests decomposing varieties into their irreducible components 
i.e.~expressing $V$ as a union $V_{1}\cup \ldots \cup V_{s}$ where each $V_{i}$ is irreducible. The irreducible 
decomposition of a variety $V$ is finite and unique. It can be obtained by finding the minimal prime ideals 
$P_{1},\ldots ,P_{s}$ over $J(V)$. In this case, the irreducible decomposition of $V$ is given by 
$V =Z(P_{1}) \cup \ldots \cup  Z(P_{s})$.

\begin{proposition}[cf.~{\cite[Thm 2.8.3.]{BocCosRoyRAG}}]
  \label{prop:irreducible-decomposition}
  Given $S\subset \mathbb{Q}[x_{1},\ldots ,x_{n}]$ let $V=Z_{\mathbb{R}}(S)$.
  If $P_{1},\ldots ,P_{s}$ are the minimal prime ideals over $\realrad{}{\langle S \rangle}$, then each prime ideal 
  $P_{i}$ is equal to $J_\mathbb{Q}(V_{i})$ for some variety $V_{i}\subset\mathbb{R}^{n}$. 
  Furthermore, $V_{1}\cup \ldots \cup V_{s}$ is the irreducible decomposition of $V$.
\end{proposition}

\section{Learning the Underlying Variety}
\label{sec:learning-variety}
\noindent
In this section, we consider the following learning problem:
\begin{itemize}
\item[(LP)] 
  \textit{Given a dataset $\Omega=(a_{1},\ldots ,a_{m})$
    of points in $[0,1]^{n}$ sampled from a variety $V \subset \mathbb{R}^n$, find that variety.}
    %\textit{ ... such that the points in $\Omega$ belong to $V$.}
\end{itemize}
\noindent
We interpret this as the optimization problem:
\begin{itemize}
\item[(OP)]
  \textit{For a fixed class of varieties $\mathcal{V}$ and an objective function
  $A:\mathcal{V}\times([0,1]^{n})^{m}\to\mathbb{R}^{+}$, given a dataset $\Omega=(a_{1},\ldots ,a_{m})$ of points in $[0,1]^{n}$ find $\argmax\limits_{V\in \mathcal{V}} A(V,\Omega)$.}
\end{itemize}  
  
One way to find a suitable objective function $A$ is by following the Bayesian machine learning paradigm \cite{barber}. Instead of working with varieties
directly, we use the observation from Proposition \ref{prop:real-variety-zero-set-single-polynomial} that every real variety $V$ can be expressed as $Z(f)$ for a single polynomial $f\in\mathbb{R}[x_{1},...,x_{n}]$.
Therefore, one can single out a class of polynomials $\Theta\subset\mathbb{R}[x_{1},\ldots ,x_{n}]$ such that each variety $V\in\mathcal{V}$ is defined by
some polynomial $f\in \Theta$. Given a posterior probability distribution $p(f|\Omega)$ over $\Theta$,
% which describes the probability of a function $f \in \Theta$ given the dataset $\Omega$,
we can define an objective function \[ A(V,\Omega):=\max\limits_{f\in\Theta,Z(f)=V}p(f|\Omega) \ . \]
The value  $A(V,\Omega)$ %$\max\limits_{f\in\Theta,Z(f)=V}p(f|\Omega)$
is to be interpreted as being proportional to the probability that $V$ is the variety from which $\Omega$ was sampled. We need to take a maximum in this definition of $A$ since different polynomials in $\Theta$ may define the same variety $V$. 
With this choice of $A$, the learning problem $\argmax\limits_{V\in\mathcal{V}}A(V,\Omega)$ reduces to the 
Maximum A Posteriori (MAP) problem  \[ \argmax\limits_{f\in\Theta}p(f|\Omega) \ .\]

Assume that we are given  a likelihood distribution $p(x|f)$ over $[0,1]^{n}$ and a prior distribution $p(f)$ over $\Theta$ so that $\frac{1}{\kappa} p(x|f)p(f)$ for a normalization constant $\frac{1}{\kappa}$. If furthermore the samples $\Omega=(a_{1},\ldots ,a_{m})$ are independent and
identically distributed (IID) \cite{barber}, then the MAP problem is explicitly given by
\[ \argmax\limits_{f\in\Theta}\frac{1}{\kappa}\prod\limits_{i=1}^{m}p(a_{i}|f)p(f)^{m} \ . \]

%To make this notion more precise, let $\mathcal{V}$ be a class of varieties and $S$ We can make this problem more precise following the Bayesian machine learning paradigm \cite{barber}. 
%For a class of varieties $\mathcal{V}$ and posterior distribution $p(V|\Omega)$, the learning problem is the Maximum A Posteriori (MAP \cite{barber}) problem: 
%\textit{Given a dataset $\Omega=\{a_{1},\ldots ,a_{m}\}$ of points in $[0,1]^{n}$, find a variety $V\in \mathcal{V}$ which maximizes $p(V|\Omega)$.} 

%Assuming the data samples are Independent and Identically Distributed (IID \cite{barber}), we define a po
The likelihood $p(x|f)$ should be roughly thought of as the probability of sampling $x\in[0,1]^{n}$ if the true underlying variety is $Z(f)$. To be robust towards noise and outliers, we allow points sampled from $Z(f)$ to be near $Z(f)$ even if they are not exactly on it. So instead of being supported on $Z(f)$, $p(x|f)$ should ideally depend on the distance of $x$ to $Z(f)$. 

%\color{blue}
\begin{remark}
In the remainder of this section we explore one way to 
obtain these probabilities. It should be emphasized that this is only one particular approach within the general theory discussed in this paper.
Further approaches and more detailed comparions will be followed up in future work.
\end{remark}
%\color{black}

The most obvious notion of distance here is the \textit{geometric distance}
\[ d_\textup{G}(x,f)=\inf\limits_{y\in Z(f)} \|x-y\|_{2} \ .  \]
However, working with $d_\textup{G}$ can be intractable, especially for optimization purposes. 
For a discussion on the complexity of this problem see \cite{Gfvert2020}.

Instead, we use a relaxation called the \textit{algebraic distance} \cite{Gander1994} given by
\[ d_\textup{A}(x,f)=|f(x)| \ . \]
By  {\L}ojasiewicz's inequality \cite[Sec.~18, Thm.~2]{LojESA}
(see also \cite{Bierstone1988,ColdingMinicozzi}),
the algebraic distance is an upper bound on the geometric distance.
More precisely, {\L}ojasiewicz's inequality entails the following result.

\begin{proposition}
  Given a polynomial $f\in \R[x_1,\ldots,x_n]$ such that $Z(f)\cap [0,1]^{n}\neq 0$,
  there exist $a,C>0$ with $0 < a < \frac{1}{2}$ such that
  \[ d_\textup{G}(x,f)\leq C \, d_\textup{A}(x,f)^a \quad \text{for all  } x\in [0,1]^{n} \ . \]
\end{proposition}
\begin{remark}\label{rem:lojasiewicz}
  \begin{enumerate}[(a)]
  \item  The compactness requirement in {\L}ojasiewicz's inequality is the main reason why we restrict our datasets to be contained 
    in the hypercube $[0,1]^{n}$. 
  \item
    By Bierstone-Milman's version of {\L}ojasiewicz's inequality \cite[Thm.~6.4]{Bierstone1988} and subanalyticity
    of the Euclidean distance to a subanalytic set \cite[page 6]{Bierstone1988} one concludes that the 
    algebraic and the geometric distances are actually equivalent. This means that in addition to the constants
    $a,C >0$ of the proposition there exist constants $c,B >0$ such that
    \[
     d_\textup{A}(x,f) \leq  B \, d_\textup{G}(x,f)^c \quad \text{for all  } x\in [0,1]^{n} \ . 
   \]
 \item
 As stated in the proposition, the constants $a,C$ depend on the polynomial $f$. If however we restrict to polynomials
 $f\in \mathbb{Z}[x_{1},...,x_{n}]$ with degree bounded by some natural $D$, then there exist global constants bounding
 the geometric distance \cite{Solern1991}. Therefore, optimizing the algebraic distance gives a reasonable proxy for
 minimizing the geometric distance.
  \end{enumerate}

\end{remark}
%\color{black}

Before proceeding further, let us introduce some additional notation. We denote monomials using multi-indices:
\[
  x^\alpha := x_{1}^{\alpha^{1}} \cdot \ldots \cdot x_{n}^{\alpha^{n}} \quad \text{for  }
  \alpha = (\alpha^{1},\ldots,\alpha^{n}) \in \N^n \ . 
\]
The degree $D $ of the monomial $ x^\alpha$ then is $D = |\alpha| :=\alpha^{1}+\ldots +\alpha^{n}$.
Correspondingly, a polynomial $f$ of degree $\leq D$ can be written as
\[
  f = \sum_{\mathsf{k=1}}^{{N}} c_{{k}}x^{\alpha_{{k}}}
    = c_{{1}}x^{\alpha_{{1}}}+\ldots +c_{{N-1}}x^{\alpha_{{N-1}}}+c_{{N}} \ ,
\]
where $N = \binom{n+D}{D}$, $c_{1},\ldots ,c_{N}\in\mathbb{R}$,
and where the multi-indices
\[
  \alpha_{1} = (\alpha_{1}^{1},\ldots, \alpha_{1}^{n}) , \ldots,
  \alpha_{N-1}  = (\alpha_{N-1}^{1},\ldots, \alpha_{N-1}^{n})
\]
are elements of $\N^n$.
For a fixed monomial ordering, $f$ can be uniquely represented as a vector in terms of its coefficients
$(c_{1},\ldots ,c_{N})$. Thus we have an isomorphism
$F:(c_{1},\ldots ,c_{N})\mapsto c_{1}x^{\alpha_{1}}+\ldots +c_{N-1}x^{\alpha_{N-1}}+c_{N}$ from the space of coefficients
$\mathbb{R}^{N}$ to
  the space $ \{f\in\mathbb{R}[x_{1},\ldots ,x_{n}]|\operatorname{Deg}(f) \leq D\}$ of polynomials of degree $\leq D$.
  Evaluating $f$ at each point of the data set  $\Omega$ results in the vector $(f(a_{1}),\ldots ,f(a_{m}))$.
  This operation can be expressed in matrix notation as
  \[ 
    \begin{pmatrix} f(a_{1}) \\ \vdots \\ f(a_{m}) \end{pmatrix}
    = U \begin{pmatrix} c_{1} \\ \vdots \\ c_{N} \end{pmatrix}  \ ,
  \]
  where $U$ is the \textit{multivariate Vandermonde matrix} \cite{Breiding2018} defined by
  $U_{ij}=x^{\alpha_{j}}(a_{i})$.

  One way to capture the inverse relationship between probability and distance is by taking the likelihood to be $p(x|f)=\frac{1}{r(f)}e^{-{d}_\textup{A}(x,f)^{2}}=\frac{1}{r(f)}e^{-f(x)^{2}}$ where $r(f)$ is the normalization factor $\int\limits_{[0,1]^{n}}e^{-f(x)^{2}} dx$. Since the algebraic distance depends on the magnitude of the coefficients, we restrict to the class	
  \[ \Theta_{D}:=\{f\in\mathbb{R}[x_{1},\ldots ,x_{n}]|\operatorname{Deg}(f)\leq D, \: \|F^{-1}(f)\|=1\} \ , \]
  where $\|F^{-1}(f)\|$ coincides with the norm on the coefficients
  $\sqrt{(c_{1},\ldots ,c_{N})(c_{1},\ldots ,c_{N})^{\trans}}$. The space $\Theta_{D}$ has an obvious
  parametrization $F|_{\sphere^{N-1}}:\sphere^{N-1}\to \Theta_{D}$ which takes a point $(c_{1},\ldots ,c_{N})$ in the sphere
  $\sphere^{N-1}$ to the polynomial $c_{1}x^{\alpha_{1}}+\ldots +c_{N-1}x^{\alpha_{N-1}}+c_{N}\in\Theta_{D}$.

The degree bound $D$ should be treated as a hyperparameter.
Choosing the hypothesis class $\Theta_{D}$ restricts us to the class $\mathcal{V}_{D}$ of varieties defined by a single polynomial of degree $\leq D$ whose coefficients have norm $1$. As with the standard machine learning set-up, there is a trade-off where setting $D$ higher leads to better approximation but higher risk of over-fitting. This issue is explored in Section \ref{sec:results}.

We define a prior distribution on $\Theta_{D}$ by $p(f)= \frac{1}{\beta}r(f)$ where $\beta$ is the normalization
constant
\[ \beta = \int_{\sphere^{N-1}} r(F(y))dy=\int_{\sphere^{N-1}}\int_{[0,1]^{n}} e^{{-(F(y)(x))}^{2}}dxdy \ . \]
This prior disfavors polynomials with low values of $r(f)$. Intuitively, such polynomials have high values of $\int_{[0,1]^{n}}|f(x)|dx$, which means that, with respect to the algebraic distance, these polynomials have a large total distance from the domain $[0,1]^{n}$. With this prior, the joint distribution simplifies to $p(x|f)p(f)=\frac{1}{\beta}e^{-f(x)^{2}}$.

With this choice of likelihood and prior distributions, the MAP problem is given by:

\noindent
\begin{align*}
\argmax\limits_{f\in\Theta_{D}}\frac{1}{\kappa}\prod\limits_{i=1}^{m}\frac{1}{\beta}e^{-f(a_{i})^{2}}
=&\argmax\limits_{f\in\Theta_{D}}\frac{1}{\kappa}(\frac{1}{\beta^{m}}e^{-\sum_{i=1}^{m}f(a_{i})^{2}}) \\
=&\argmax\limits_{f\in\Theta_{D}}\ \log( \frac{1}{\kappa}(\frac{1}{\beta^{m}}e^{-\sum_{i=1}^{m}f(a_{i})^{2}}))\\
=&\argmax\limits_{f\in\Theta_{D}}\ -\log(\kappa)-\log(\beta^{m})-\sum_{i=1}^{m}f(a_{i})^{2} \\
=& \argmin\limits_{f\in\Theta_{D}} \ \sum_{i=1}^{m}f(a_{i})^{2}\ .
\end{align*}
\noindent
This is equivalent to the problem
\[ \argmin\limits_{c \in \sphere^{N-1}} \ (c_{1},\ldots ,c_{N})\ U^{\trans}U\ (c_{1},\ldots ,c_{N})^{\trans}  \]
\noindent
which is a convex Quadratically Constrained Quadratic Program (QCQP) \cite{Convex}. The solutions are the normalized elements of the eigenspace $E_{\lambda}$ where $\lambda$ is the smallest eigenvalue of $U^{\trans}U$.

%\color{blue}
\begin{remark}
More generally, we can think of likelihood distributions of the form $\frac{e^{-\lambda f(x)^{2}}}{r(g,\lambda)}$ where $\lambda$ accounts for a variance-like term 
with $\lambda \in (0,\infty)$. In this case, by taking derivatives with respect to $\lambda$, optimization with respect to $\lambda$ is obtained by solving the following equation for $\lambda$:
\[ 
  \frac{\int_{\sphere^{N-1}}\int_{[0,1]^{n}} (F(y)(x))^{2}e^{{-\lambda(F(y)(x))}^{2}}dxdy}{\int_{\sphere^{N-1}}\int_{[0,1]^{n}} e^{{-\lambda(F(y)(x))}^{2}}dxdy}=\frac{\sum_{i=1}^{m}f(a_{i})^{2}}{m} \ .
\]
Intuitively, this says that $\lambda$ should be chosen so that over all the distributions coming from $\Phi_{D}$, the average variance agrees with the empirical variance $\frac{\sum_{i=1}^{m}f(a_{i})^{2}}{m}$.

On the other hand, optimization with respect to $f$ is again given by 
\[ \argmin\limits_{f\in\Theta_{D}}\sum_{i=1}^{m}f(a_{i})^{2} \ . \]
We are currently not aware of a computationally efficient way to infer $\lambda $ from the data, so we leave a more detailed consideration
of this issue for later work.
\end{remark}
%\color{black}

%m\log(\int_{\sphere^{N-1}}\int_{[0,1]^{n}} e^{{-\lambda(F(y)(x))}^{2}}dxdy)+\lambda\sum_{i=1}^{m}f(a_{i})^{2}$. Taking the derivative wrt $\lambda$ obtains $m\frac{\int_{\sphere^{N-1}}\int_{[0,1]^{n}} \lambda(F(y)(x))^{2}e^{{-\lambda(F(y)(x))}^{2}}dxdy}{\int_{\sphere^{N-1}}\int_{[0,1]^{n}} e^{{-\lambda(F(y)(x))}^{2}}dxdy}=\frac{\sum_{i=1}^{m}f(a_{i})^{2}}{m}$. (Assuming we can pull the derivative). On the other hand, optimizing wrt $c_{1},...,c_{N}$ is independent of the integral term and is thus given by $\argmax\limits_{f\in\Theta_{D}}\ -\log(\kappa)-\log(\beta^{m})-\sum_{i=1}^{m}f(a_{i})^{2}
%= \argmin\limits_{f\in\Theta_{D}} \ \sum_{i=1}^{m}f(a_{i})^{2}$.}

In the case where $\lambda>0$, the matrix $U^{\trans}U$ is positive definite, in which case the QCQP is strictly convex \cite{Convex},
so there is essentially one unique MAP solution $\hat{f}$. In the case where $\lambda=0$, the MAP solutions are the normalized elements of 
$E_{\lambda}= \ker (A^{\trans}A)$. Here, every MAP solution $\hat{f}$ vanishes exactly on $\Omega$ and $p(\hat{f}|\Omega)=\frac{1}{\kappa\beta^{m}}$. 

Therefore, under the assumptions that the hypothesis class is $\mathcal{V}_{D}$, the objective function is $A(V,\Omega)=\max\limits_{f\in\Theta,Z(f)=V}p(f|\Omega)$, and the posterior distribution is $p(f|\Omega)=\frac{1}{\kappa}(\frac{1}{\beta^{m}}e^{-\sum_{i=1}^{m}f(a_{i})^{2}})$ on $\Theta_{D}$, the solutions to the learning problem are the varieties $Z(\hat{f})$ for every normalized element $\hat{f}\in E_{\lambda}$. We call this the \emph{MAP model} and  summarize it in the algorithm below.
\vspace{2mm}

\begin{algorithm}[H]\label{MAP-Model}
\SetAlgoLined
 \textbf{Input:} a dataset $\Omega\in ([0,1]^{n})^{m}$ and a degree bound $D$. \\
 \textbf{Output:} a polynomial $\hat{f}$ in $\Theta_{D}$, such that $Z(\hat{f})$ solves the learning problem under the above assumptions.\\[2mm]
 Fix an ordering and list all homogeneous monomials $x^{\alpha_{1}}, \ldots ,x^{\alpha_{N}}$ of degree $\leq D$.\\
 Compute the multivariate Vandermonde matrix $U_{ij}=x^{\alpha_{j}}(a_{i})$.\\
 Find the smallest eigenvalue $\lambda$ of $U^{\trans}U$ and its corresponding eigenspace $E_{\lambda}$.\\
 \textbf{return:} any normalized element $\hat{f}$ of $E_{\lambda}$
 \caption{MAP Model}
\end{algorithm}

\subsection{Expanding on the MAP Model}
One way to refine the result for the case $\lambda = 0$ is to enlarge our hypothesis class. If $\lambda = 0$ and $f_{1},\ldots ,f_{k}$ is a normalized basis for $\ker (U^{\trans}U)$, then $\hat{f}:=f_{1}^{2}+\ldots +f_{k}^{2}$ is a degree $2D$ polynomial whose zero-set satisfies
\[ Z(\hat{f}) = Z(\{f_{1},\ldots ,f_{k}\})=\bigcap\limits_{f\in\ker(U^{\trans}U)} Z(f) \ . \]
 That is, $\hat{f}$ defines the smallest variety given by a set of polynomials of degree $\leq D$. This method of taking intersections changes the hypothesis class and may no longer yield an MAP solution over $\Theta_{2D}$. However, this method yields a less redundant variety than the MAP solutions over $\Theta_{D}$ without the need to preform any further optimization. We call this the \emph{intersected MAP model}. It is summarized in the algorithm below.
\vspace{2mm}

\begin{algorithm}[H] \label{Intersected-MAP-Model}
\SetAlgoLined
 \textbf{Input:} a dataset $\Omega\in ([0,1]^{n})^{m}$ and a degree bound $D$. \\
 \textbf{Output:} a polynomial $\hat{f}$ in $\Theta_{2D}$, such that $Z(\hat{f})$ is the intersection of all solutions in $\mathcal{V}_{D}$ to the learning problem under the previous assumptions.\\[2mm]
 Fix an ordering and list all homogeneous monomials $x^{\alpha_{1}},\ldots ,x^{\alpha_{N}}$ of degree $\leq D$. \\
 Compute the multivariate Vandermonde matrix $U_{ij}=x^{\alpha_{j}}(a_{i})$.\\
 Find the smallest eigenvalue $\lambda$ of $U^{\trans}U$ and its corresponding eigenspace $E_{\lambda}$.\\
  \eIf{$\lambda>0$}{
    \textbf{return:} the (essentially) unique normalized element $\hat{f}$ of $E_{\lambda}$
   }{
   Find a basis $f_{1},\ldots ,f_{k}$ for $\ker (U^{\trans}U)$.\\
   \textbf{return:} $\hat{f}:= f_{1}^{2}+\ldots +f_{k}^{2}$.
  }
 \caption{Intersected MAP Model}
\end{algorithm}

\section{Algebraic Computations}
\label{sec:algebraic-computations}
Assume that $Z(f)$ for $f$ a polynomial in $n$ real variables
% with $f\in \K[x_1,\ldots , x_n]$ and $\K=\R$ or $\K=\Q$
is the true underlying variety for the data set $\Omega$. We can use tools from commutative algebra to reveal information
about the geometry of $Z(f) \subset \R^n$. This analysis can be automated with the use of Gr\"obner bases
which are particular and computationally powerful types of generating sets for ideals 
$I \subset  \K[x_1,\ldots , x_n]$. 
%a Gr\"obner basis of $I$ is a particular and computationally powerful type of generating set for the ideal $I$.
We do not state the precise definition of a Gr\"obner basis here, but refer the reader to the excellent exposition
in \cite[Sec.\ 1.6]{10.5555/1557288}. Moreover, \cite{10.5555/1557288} provides a general overview on
computational commutative algebra.
To use Gr\"obner basis methods in a computer algebra system like SINGULAR \cite{SINGULAR}, we have to change the base field from $\mathbb{R}$ to $\mathbb{Q}$. If $f$ was obtained through a numerical procedure such as the MAP learning model, then the floating point coefficients of $f$ can be interpreted as rational numbers.

With Gr\"obner basis methods one can compute a generating set $f_{1},\ldots ,f_{k}$ for the ideal
\[ \realrad{\Q}{\langle f\rangle}\subset \mathbb{Q}[x_{1},\ldots ,x_{n}] \ .\] 
Using this generating set we can construct the ring $\mathbb{Q}[x_{1},\ldots ,x_{n}]/\realrad{\Q}{\langle f\rangle}$ which is a subring of the coordinate ring $\mathbb{R}(Z_{\mathbb{R}}(f))$. 
By Proposition \ref{prop:dimension-geometric-algebraic}, 
\[ \operatorname{dim}(Z_{\mathbb{R}}(f))=\operatorname{dim}(\mathbb{Q}[x_{1},\ldots ,x_{n}]/\realrad{\Q}{\langle f\rangle})\]
which can be computed using Gr\"obner bases. Similarly, we can compute the minimal prime ideals over $\realrad{\Q}{\langle f\rangle})$. By Proposition \ref{prop:irreducible-decomposition}, these are the ideals $J_{\mathbb{Q}}(V_{1}),\ldots ,J_{\mathbb{Q}}(V_{s})$,
where $V_{1},\ldots ,V_{s}$ are the irreducible components of $V$.

It should be noted, however, that even over $\mathbb{Q}$, Gr\"obner basis computations are in general very costly and may only be feasible for small $n$ and $D$, hence the need for numerical computations. For more details on the complexity of finding Gr\"obner bases see \cite{HUYNH1986196}.

\begin{example}
%[cite \textcolor{red}{https://service.mathematik.uni-kl.de/ftp/pub/Math/Singular/doc/Papers/diplom\textunderscore spang.pdf}]

We can apply these concepts to the variety $V$ from Example \ref{ex:union-sphere-plane} using the following SINGULAR code.

\begin{verbatim}
  // Define R = QQ[x,y,z] with lexicographic ordering.
  ring R = 0,(x,y,z),lp; 
  poly f = ((x-1/2)^2 + (y-1/2)^2 + (z-1/2)^2 - 1/4)*(x-y);
  ideal I = f;
  LIB "realrad.lib";
  ideal I2 = realrad(I);
  size(reduce(I,I2));
  //->0
  size(reduce(I2,I));
  //->0
  LIB "primdec.lib";
  minAssGTZ(I2);
  //->[1]:
  //-> _[1]=2x2-2x+2y2-2y+2z2-2z+1
  //->[2]:
  //-> _[1]=x-y
  dim(I2);
  //->2
\end{verbatim}

First we define the ideal $I=\langle((x-\frac{1}{2})^{2} + (y-\frac{1}{2})^{2} + (z-\frac{1}{2})^{2} - \frac{1}{4})\cdot(x-y)\rangle$ over $\mathbb{Q}$. We then compute the real radical using the library \texttt{realrad.lib} \cite{realrad} and
observe that in fact $I=\realrad{\Q}{I}$. As expected, we find that $\dim (\mathbb{Q}[x_{1},\ldots ,x_{n}]/\realrad{\Q}{I})=2$ which coincides with $\dim (V)$. Using the library \texttt{primdec.lib} we also determine the minimal primes above $\realrad{\Q}{I}$ to be the vanishing ideal of the sphere of radius $\frac{1}{2}$ centered around $(\frac{1}{2},\frac{1}{2},\frac{1}{2})$ and the
vanishing ideal of the plane $x-y$. This agrees with the decomposition of $V$.

\end{example}

\section{Numerical Computations}
\label{sec:singular-heuristics}
Again, assume that $Z(f)$ is the true underlying variety for $\Omega$. We can also use numerical methods to study the geometry of $Z(f)$ by sampling new points from $Z(f)$ and working directly with those samples. This approach has the advantage of being computationally more tracktable than the algebraic computations relying on Gr\"obner bases.

One reason to obtain a new sample set $\Omega'$ from $Z(f)\cap[0,1]^{n}$ is that if $Z(f)$ is only an approximate fit, then the sample $\Omega'$ will reflect the geometry of $Z(f)$ more closely than $\Omega$. Hence, if we are specifically studying the model $Z(f)$, generating a sample set $\Omega'$ can lead to more accurate results.

First notice that the likelihood distributions $p(x|f)$ on $[0,1]^{n}$ can be used to generate data. This can be achieved using \textit{rejection sampling} which is outlined in Algorithm \ref{alg:rejection-sampling} . This sampling process reveals the underlying assumptions that the model makes about how the original data set $\Omega$ was generated and how noise was introduced.

However, capturing the model's assumptions also captures the noisy process through which $\Omega$ was supposedly generated. If the likelihood distribution depends on the algebraic distance, such as the distribution used in Section \ref{sec:learning-variety}, then this noise can be avoided by fixing a small $\eta > 0$ and accepting a point $a\in[0,1]^{n}$ if and only if $d_{A}(a,Z(f))=|f(a)|<\eta$. We call this \textit{direct sampling} and summarize it in Algorithm \ref{alg:direct-sampling}.

Setting a smaller $\eta$ threshold reduces the sampling noise, however this comes at the cost of increasing the efficiency of the sampling. This still works reasonably well for low values of $n$, but it does not scale well to higher dimensions. For higher dimensions, a more effective sampling method is given in \cite{8999343}. Alternatively, the variety could be sampled using Homotopy Continuation \cite{10.1007/978-3-319-96418-8_54}, which is a numerical method for computing the zero-set of a system of polynomials. In Homotopy Continuation, one begins with a simple system of polynomials whose roots are known, and then defines a homotopy, i.e.\ a continuous deformation, from the simple system to the system of polynomials that one is trying to solve. This method relies on tracking the paths that the roots take as the homotopy is being applied.

\begin{algorithm}[h]
  
\SetAlgoLined
 \textbf{Input:} a polynomial $f$, a likelihood distribution $p(x|f)$, and a target number of samples $m$.\\
 \textbf{Output:} a set $\Omega'$ of $m$ points sampled from $Z(f)\cap[0,1]^{n}$ according to $p(x|f)$.\\
 Initialize $\Omega'$ to $\emptyset$.\\
 \While{$|\Omega'|<m$}{
  Draw a random point $a$ from the uniform distribution on $[0,1]^{n}$.\\
  Draw a random number $\alpha$ from $[0,1]$.\\
  \eIf{$\alpha < p(a|f))$}{
    \textbf{Accept:} $\Omega'\xleftarrow{}\Omega'\cup \{a\}$.
   }{
   \textbf{Reject}.
  }
 }
 \textbf{return:} $\Omega'$.
 
 \caption{Rejection Sampling}\label{alg:rejection-sampling}
 
\end{algorithm}

\begin{algorithm}[h]

\SetAlgoLined
 \textbf{Input:} a polynomial $f$ and a target number of samples $m$.\\
 \textbf{Output:} a set $\Omega'$ of $m$ points sampled from $Z(f)\cap[0,1]^{n}$.\\
 Initialize $\Omega'$ to $\emptyset$.\\
 \While{$|\Omega'|<m$}{
  Draw a random point $a$ from the uniform distribution on $[0,1]^{n}$.\\
  \eIf{$|f(a)|<\eta$}{
   \textbf{Accept:} $\Omega'\xleftarrow{}\Omega'\cup \{a\}$.
   }{
   \textbf{Reject}.
  }
 }
 \textbf{return:} $\Omega'$.
 \caption{Direct Sampling}\label{alg:direct-sampling} 
\end{algorithm}

Let $\Omega'$ be a set of samples from $Z(f)\cap [0,1]^{n}$ obtained using direct sampling with an accuracy threshold of $\eta$. Under the hypothesis that the distribution $p(x|f)$ depends on $d_{A}(x,f)$, we propose a method for finding the points in $\Omega'$ near $\operatorname{Sing}(Z(f))$. First, by Theorem \ref{thm:singular-locus-hypersurface}, if $V$ is $n-1$ dimensional and $f$ is irreducible, then every point $b\in [0,1]^{n}$ which lies in
the singular locus $\operatorname{Sing}(Z(f))$ satisfies
\[  f(b)=\partial_{x_{1}}f(b)=\ldots =\partial_{x_{n}}f(b)=0 \ .\]
By continuity of the map $\| \nabla f(x)\|_{2}=\sqrt{(\partial_{x_{1}}f(x))^{2}+\ldots +(\partial_{x_{n}}f(x))^{2}}$
there exists for every $\varepsilon >0$ % $\varepsilon >\eta$ 
an open neighborhood $U$ of $b$ such that all points
$a\in U\cap \Omega'$ satisfy $\| \nabla f(a)\|_{2}<\varepsilon$.

Note, however, that the converse of this need not be true meaning that a point $a\in\Omega'$ satisfying $\| \nabla f(a)\|_{2}<\varepsilon$ is
not necessarily near a point $b\in\operatorname{Sing}(Z(f))$. For example, consider the polynomial $f \in \R[x,y]$ given by 
$f(x,y)=x^2$.  The gradient $\nabla f $ vanishes exactly on the line $Z(x)=\{ (x,y)\in \R^2\mid x =0\}$
which is smooth and coincides with the variety $Z(f)$. Hence the zero locus of the gradient
is nowhere close to the singular locus of the variety $Z(f)$ which is empty by smoothness of $Z(f) =Z(x)$.

Nevertheless, assuming the converse appears to be justified for computational purposes and it provides a powerful heuristic method for detecting singularities. More specifically, we can heuristically assume that regardless of the dimension of $V$ and the reducibility of $f$, if $\varepsilon>0$ is small enough, then $\{a\in \Omega'|\|\nabla f(a)\|_{2}<\varepsilon\}$ is a set of points at or near $\operatorname{Sing}(Z(f))\cap [0,1]^{n}$. We will denote this set by
$\operatorname{Sing} (\Omega')$ and call the described method the \textit{singularity heuristic}. Clearly, the
accuracy of this method depends on the magnitudes of $\eta$ and $\varepsilon$ and the density of the sample
set $\Omega'$. We explore the efficacy of the singularity heuristic in Section \ref{sec:results}.

\begin{algorithm}[h]
\SetAlgoLined
 \textbf{Input:} a polynomial $f$, a set of samples $\Omega'$ from $Z(f)\cap[0,1]^{n}$, and a singularity threshold  $\varepsilon$. \\
 \textbf{Output:} $\operatorname{Sing}(\Omega')$, a subset of $\Omega'$, heuristically assumed to be near $\operatorname{Sing}(Z(f))$ \\
 Initialize $\operatorname{Sing}(\Omega')$ to $\emptyset$.\\
 Find the partial derivatives $\partial_{x_{1}}f(x),\ldots ,\partial_{x_{n}}f(x)$,\\
 \For{$a\in \Omega'$}{

  \eIf{$\|\nabla f(a)\|_{2}<\varepsilon$}{
   \textbf{Accept} $a$  $\operatorname{Sing}(\Omega')\xleftarrow{}\operatorname{Sing}(\Omega')\cup \{a\}$.
   }{
   \textbf{Reject}.
  }
 }
 \textbf{Return:} $\operatorname{Sing}(\Omega')$.
 \caption{Singularity Heuristic}
\end{algorithm}

Another crucial numerical task is to test if two varieties $Z(f)$ and $Z(g)$ are equal. 
Note that the equality
\begin{equation}
  \label{eq:equalityzerosets}
  Z(f)=Z(g)
\end{equation}
does in general not entail the  polynomials $f$ and $g$ to be equal as for example the choice $f = x-y$ and $g = (x^2 + y^2) \cdot (x-y)$ shows.
Furthermore, testing the equality \eqref{eq:equalityzerosets} is obviously not possible if one of the polynomials, say $g$, is unknown and we only have access
to samples $\Lambda$ from $Z(g)\cap [0,1]^{n}$. One solution is to work entirely in terms of samples. Namely, 
take a set of samples $\Omega'$ from $Z(f)\cap [0,1]^{n}$ and directly compare $\Lambda$ 
with $\Omega'$. To compare $\Lambda$ with $\Omega'$, we use the Wasserstein distance \cite{ruschendorf1985wasserstein}, which measures the cost of the 
optimal transport taking the point cloud $\Lambda$ to the point cloud $\Omega'$. If we have $N$ samples from $\mathbb{R}^n$ this is simply
\[
  W (\Omega', \Lambda) = \min\limits_{\pi \in \operatorname{S}_N} \left\{\sum_{i=1}^N\big\|\Omega'_i - \Lambda_{\pi(i)}\big\|_2 \right\} \ ,
\]
where $\operatorname{S}_N$ denotes the group of permutations of the integers $\{1,2, \ldots , N\}$.  We can similarly apply 
the Wasserstein distance in order to compare a set of samples from $\operatorname{Sing}(Z(g))\cap [0,1]^{n}$ with the set
$\operatorname{Sing}(\Omega')$.

\section{Results}
\label{sec:results}
In this section, we test the MAP model from Section \ref{sec:learning-variety} and the singularity heuristic from Section \ref{sec:singular-heuristics}.

Starting with the variety $V$ from Example \ref{ex:union-sphere-plane}, we used the parametric form of $V$ to produce a sample set $\Omega$ consisting of $1600$ points sampled from $V\cap [0,1]^{3}$ (Figure \ref{fig:7.1a}). We also used the parametric form of $\operatorname{Sing}(V)$ to
produce a sample set, call it $\operatorname{Sing}(\Omega)$, of $400$ points sampled from $\operatorname{Sing} (V)$ (Figure \ref{fig:7.1d}). Plotted in  Figure \ref{fig:7.1c} is the learned variety.

\begin{figure}
\centering
\begin{subfigure}{.3\textwidth}
\centering
\includegraphics[scale=.25]{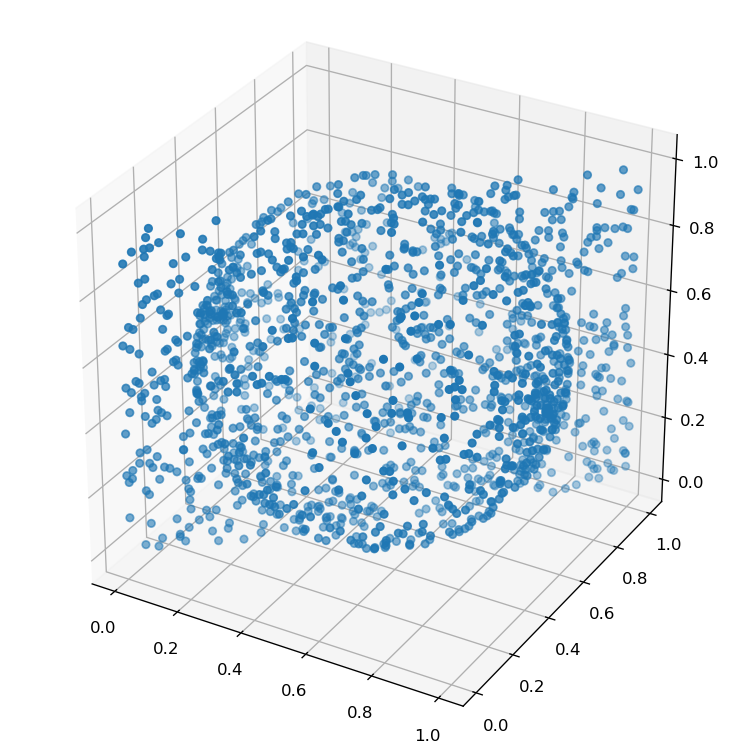}
\caption{}\label{fig:7.1a}
\end{subfigure}
\begin{subfigure}{.3\textwidth}
\centering
\includegraphics[scale=.25]{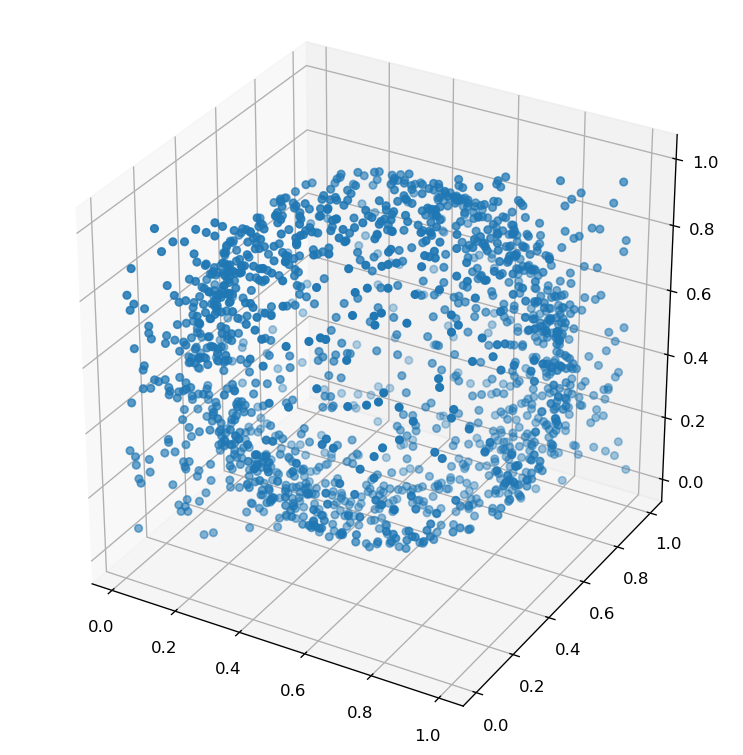}
\caption{}\label{fig:7.1b}
\end{subfigure}
\begin{subfigure}{.3\textwidth}
\centering
\includegraphics[scale=0.25]{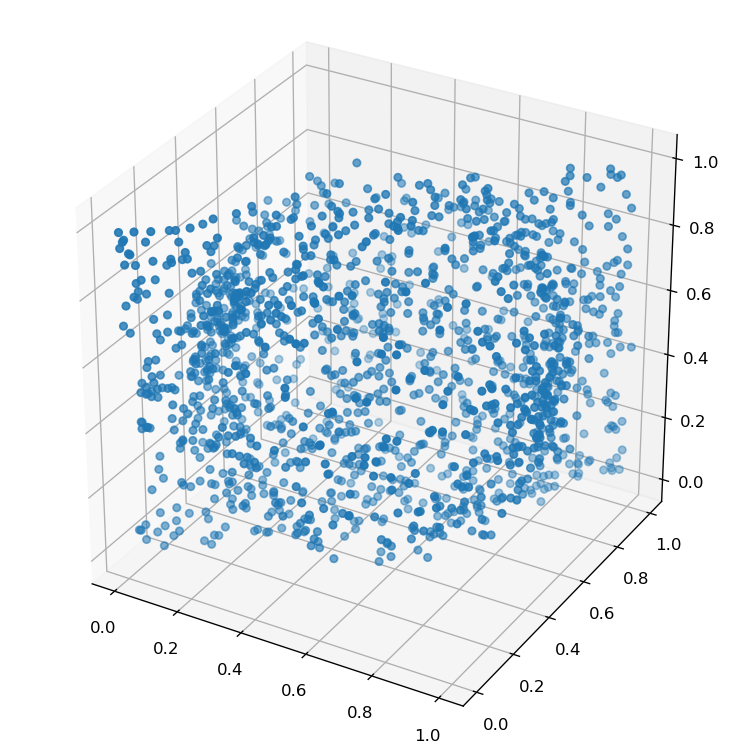}
\caption{}\label{fig:7.1c}
\end{subfigure}
\caption{Samples from $V$ and from the learned variety}\label{fig:7sampling-variety}
\end{figure}
\begin{figure}
\centering
\begin{subfigure}{.3\textwidth}\centering
\includegraphics[scale=.25]{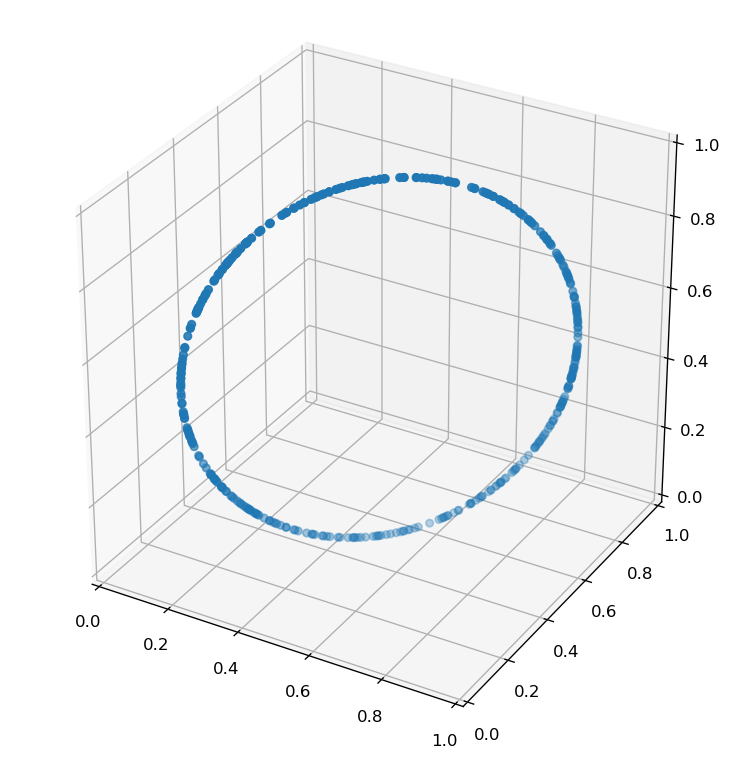}
\caption{}\label{fig:7.1d}
\end{subfigure}
\begin{subfigure}{.3\textwidth}\centering
\includegraphics[scale=.25]{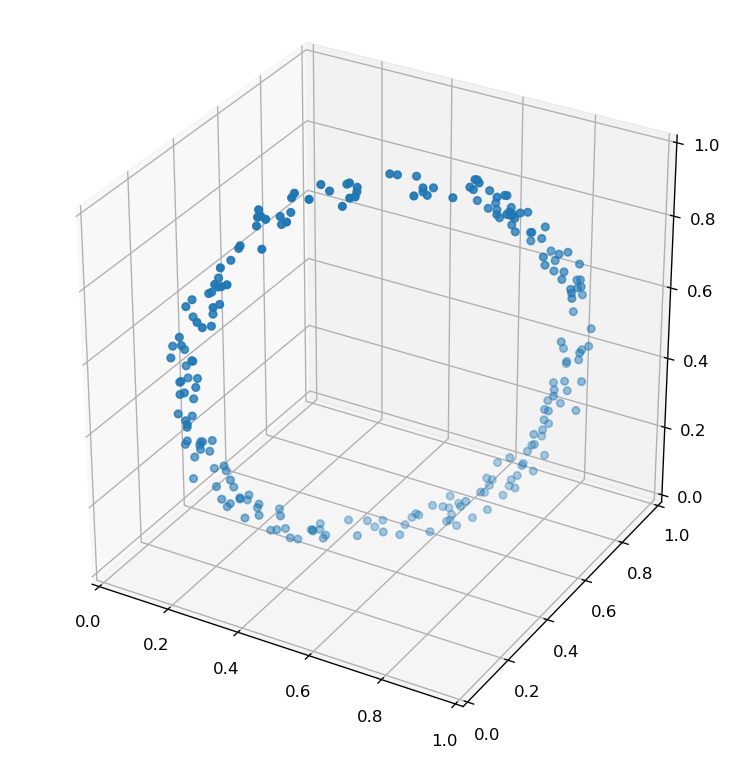}
\caption{}\label{fig:7.1e}
\end{subfigure}
\begin{subfigure}{.3\textwidth}
\end{subfigure}\caption{Samples from the singular locus of $V$ and results from the singularity heuristic}\label{fig:7singular-locus}
\end{figure}
We applied the MAP model with $D=3$. After rounding and scaling, we obtained exactly the target polynomial from Example \ref{ex:union-sphere-plane} 
\begin{eqnarray*}
  x^{3}-x^{2}y-x^{2}+xy^{2}+xz^{2}-xz+\frac{1}{2}x-y^3+y^2-yz^{2}+yz-\frac{1}{2}y \!\!&=& \\
  = \left((x-\frac{1}{2})^{2}+(y-\frac{1}{2})^{2}+(z-\frac{1}{2})^{2}-\frac{1}{4}\right)
  \cdot(x-y) \!\!&=&\!\! f \ .
\end{eqnarray*}

%\[ f=x^{3}-x^{2}y-x^{2}+xy^{2}+xz^{2}-xz+\frac{1}{2}x-y^3+y^2-yz^{2}+yz-\frac{1}{2}y \ . \]
%Which is exactly the target polynomial
%\[ ((x-\frac{1}{2})^{2}+(y-\frac{1}{2})^{2}+(z-\frac{1}{2})^{2}-\frac{1}{4})\cdot(x-y) \ . \]

By direct sampling with $\eta=0.001$ we produced a data set $\Omega'$ consisting of $1600$ points sampled from $Z(f)\cap[0,1]^{3}$
(Figure \ref{fig:7.1b}).
Using these samples, we applied the singularity heuristic with value $\varepsilon = 0.02$ and obtained a set $\operatorname{Sing} (\Omega')$
consisting of $232$ points (Figure \ref{fig:7.1e}).

To test the MAP model under more general conditions, we used the dataset $\Omega$ as above and tested the MAP model for different values of $D$. For a more quantitative picture, we computed the Wasserstein distance between $\Omega$ and $1600$ points obtained through direct sampling from $Z(f)\cap[0,1]^{3}$ for each learned polynomial $f$ and for different values of $\eta$. We repeated each test $3$ times. The average values are given in Figure \ref{fig:7.2a}, where $\log_{10} (\eta)$ is the variable along the horizontal axis and the Wasserstein distance
$W(\Omega,\Omega')$ is the variable along the vertical axis.  

\begin{figure}%[H]
\centering
\begin{subfigure}{0.32\textwidth}\centering
\includegraphics[scale=0.4]{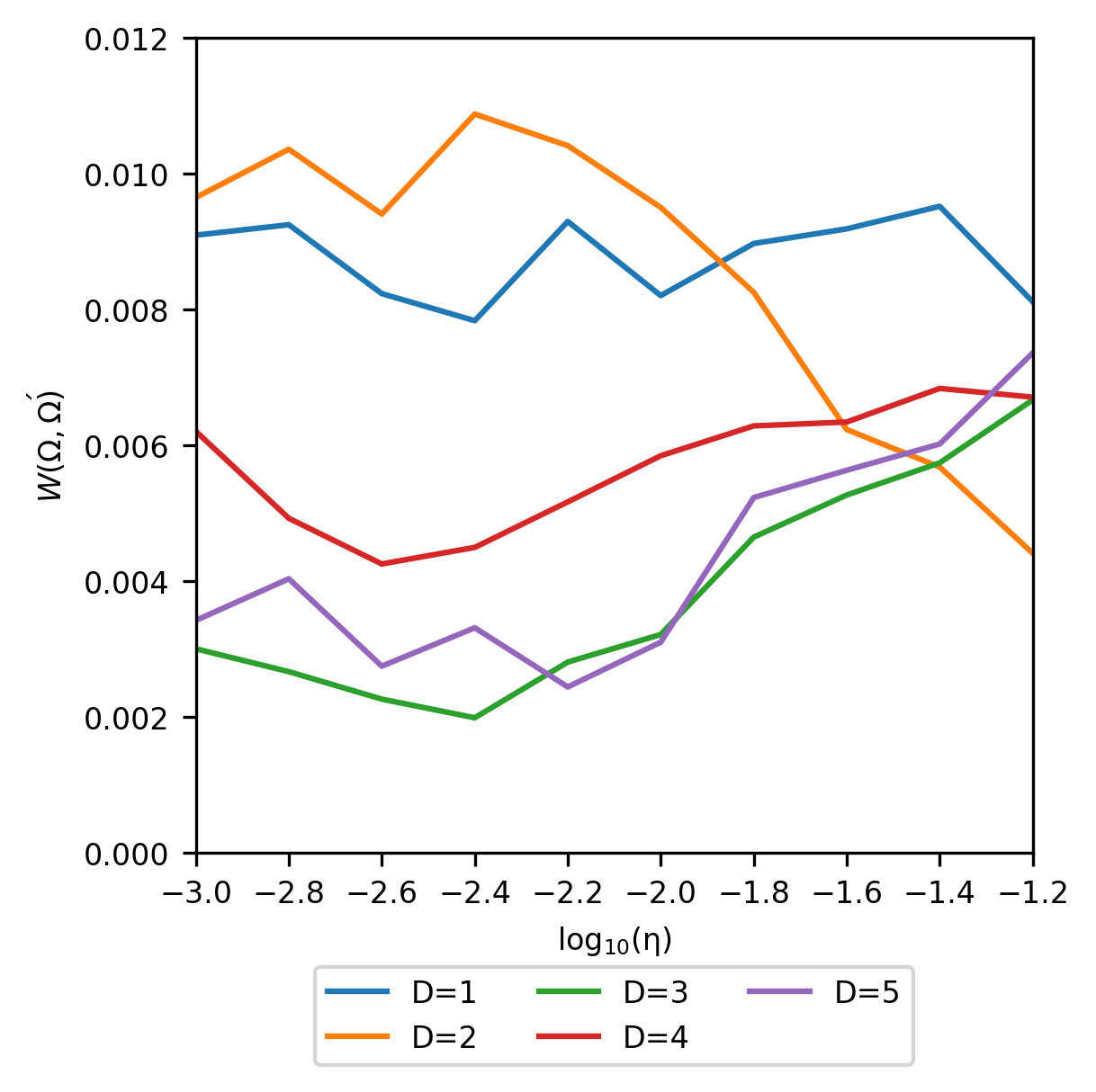}
\caption{}\label{fig:7.2a}
\end{subfigure}
\begin{subfigure}{.32\textwidth}\centering
\includegraphics[scale=0.4]{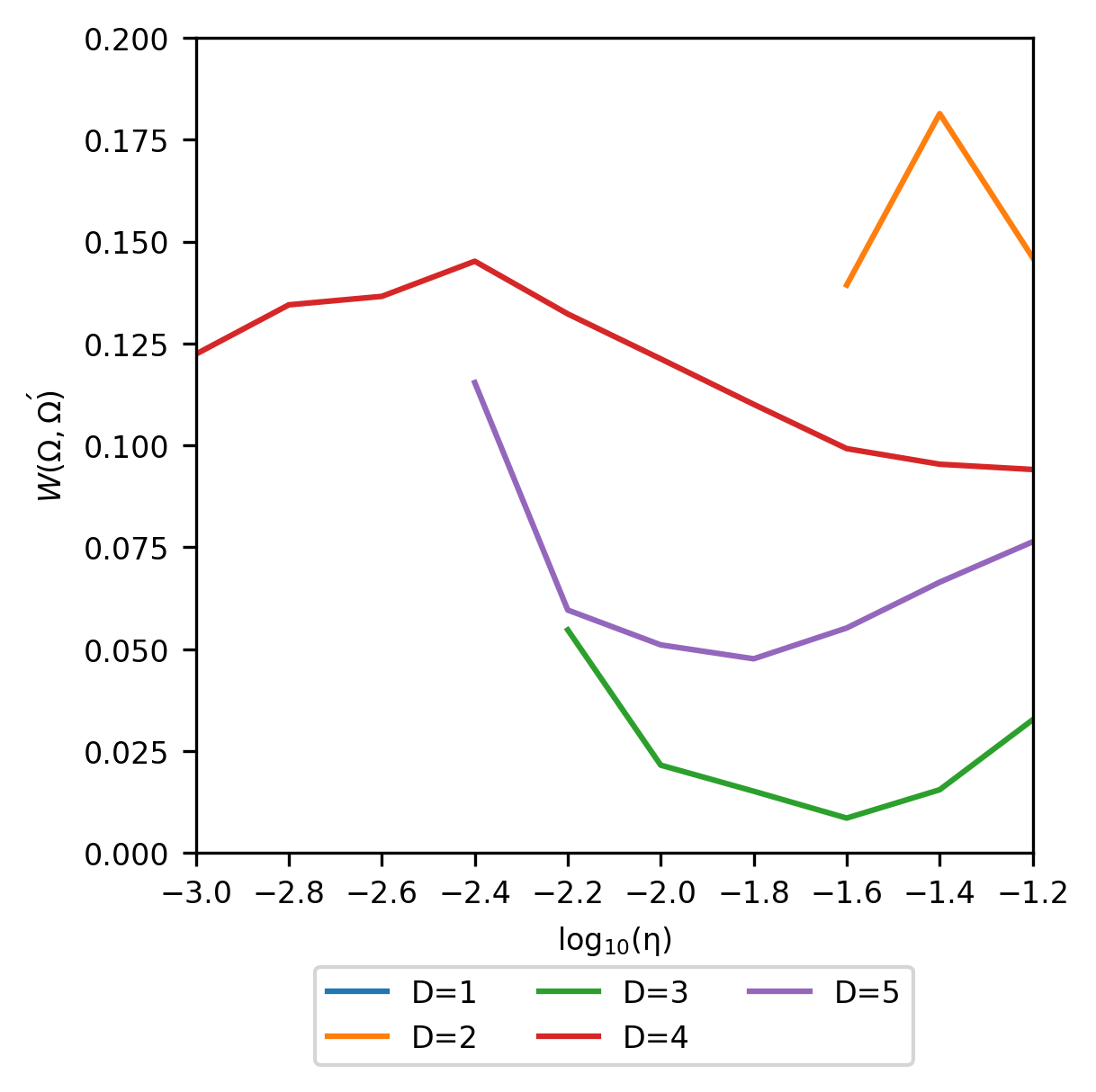}
\caption{}\label{fig:7.2b}
\end{subfigure}
\begin{subfigure}{.32\textwidth}\centering
\includegraphics[scale=0.35]{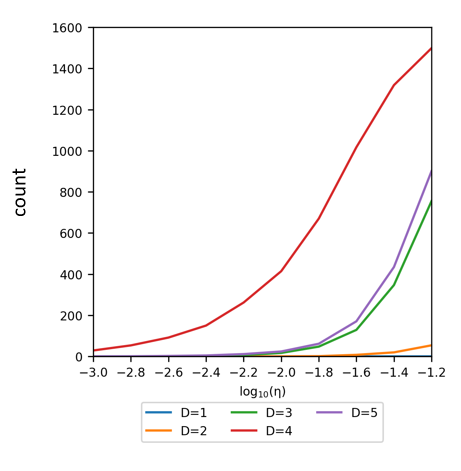}
\caption{}\label{fig:7.2c}
\end{subfigure}\caption{MAP model and singularity heuristic performance (noise-free)}%\\(Wasserstein distance in vertical axis of (a) and (b))}
\end{figure}

\begin{figure}%[H]
\centering
\begin{subfigure}{0.32\textwidth}\centering
\includegraphics[scale=0.4]{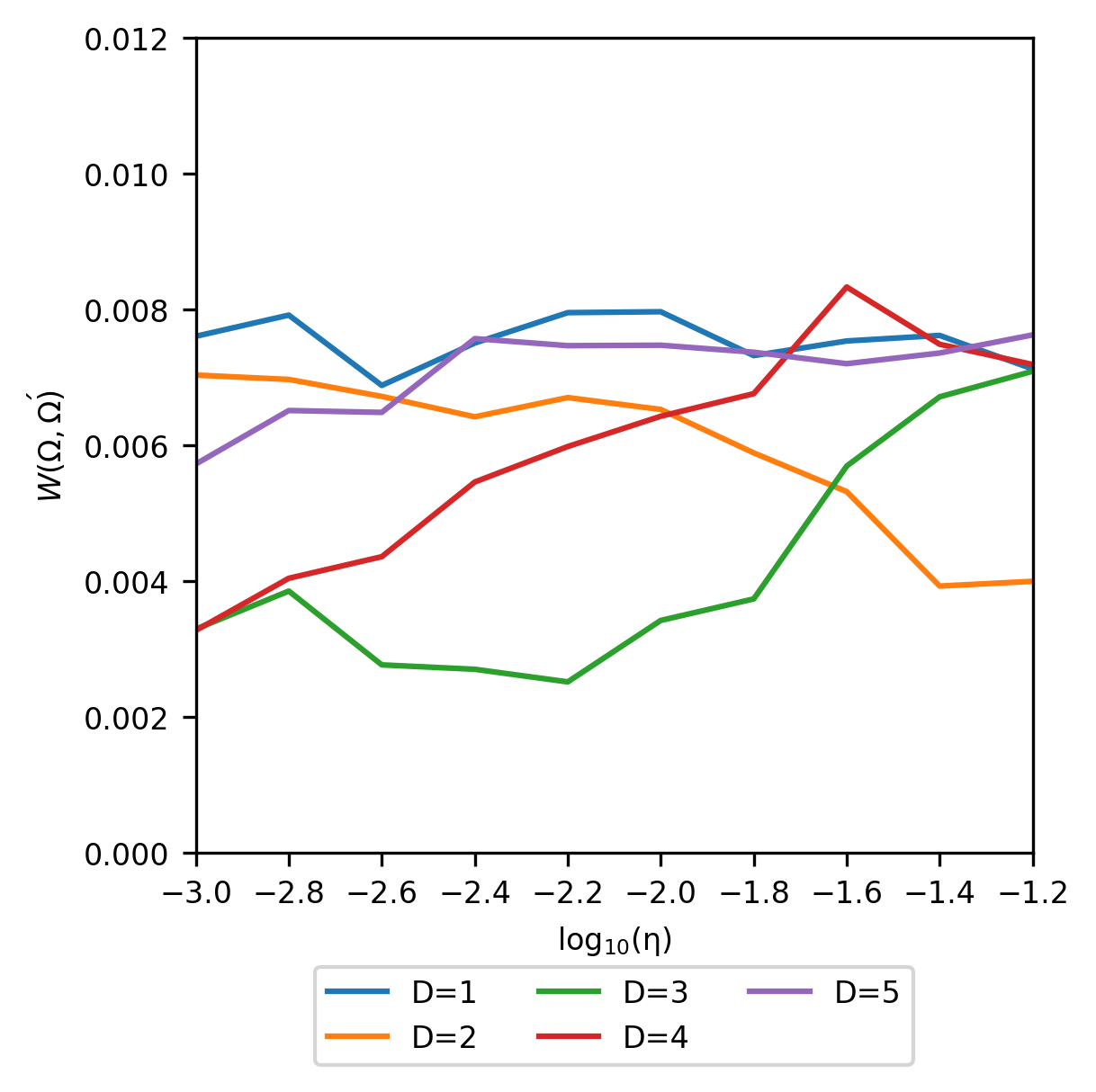}
\caption{}\label{fig:7.2d}
\end{subfigure}
\begin{subfigure}{.32\textwidth}\centering
\includegraphics[scale=0.4]{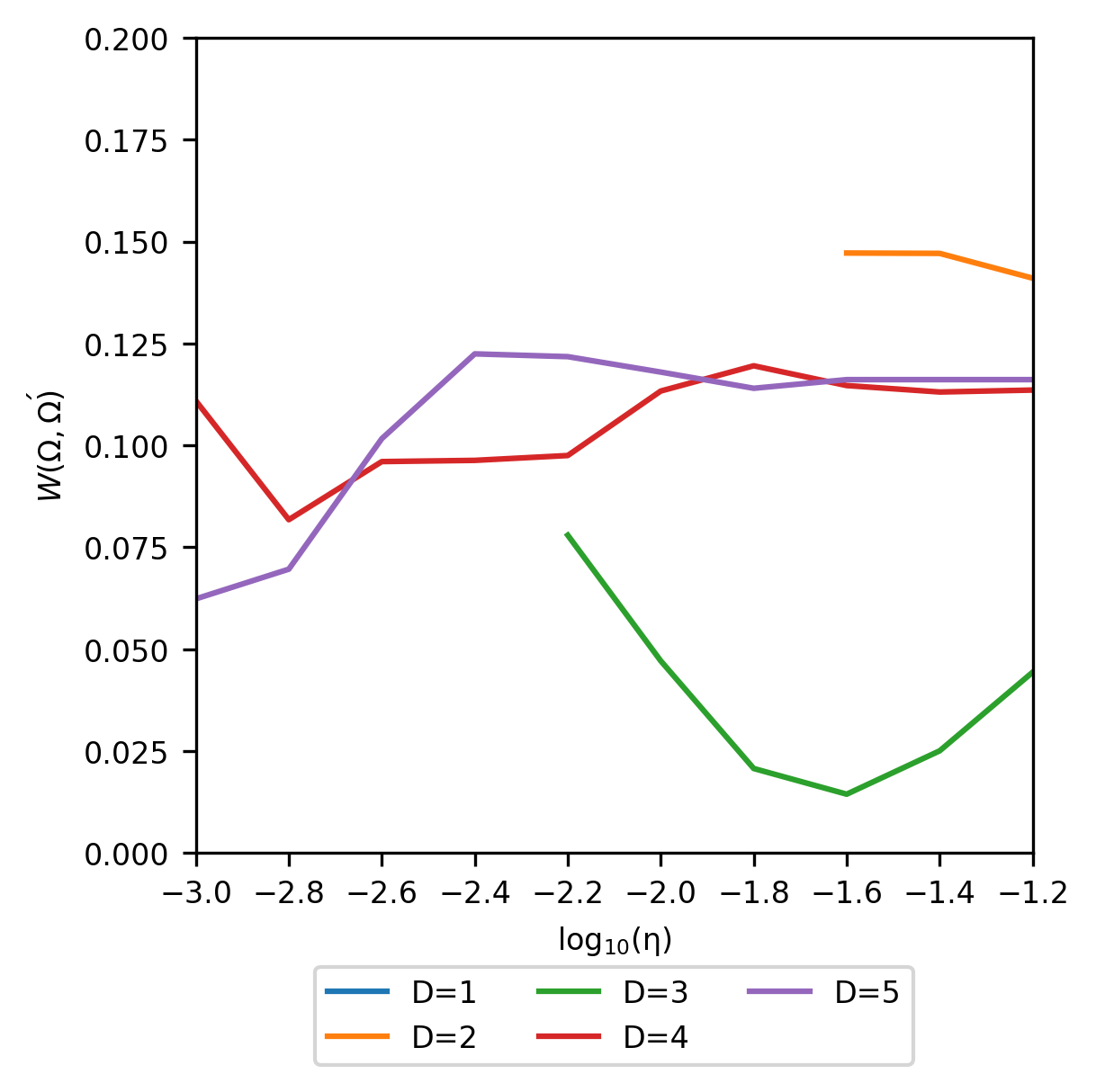}
\caption{}\label{fig:7.2e}
\end{subfigure}
\begin{subfigure}{.32\textwidth}\centering
\includegraphics[scale=0.35]{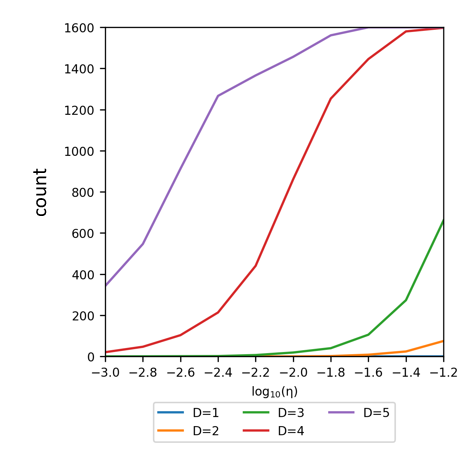}
\caption{}\label{fig:7.2f}
\end{subfigure}\caption{MAP model and singularity heuristic performance (noise added)}
\end{figure}

As one can see, increasing $\eta$ largely results in an increase in the Wasserstein distance. However, very low values of $\eta$ seem to perform worse than higher values. One explanation for this is that direct sampling does not produce a uniform set of samples over $Z(f)\cap[0,1]^{3}$. This can be seen for example in Figure \ref{fig:7.1b}. Nonetheless, we see that the overall minimal distance is attained with $D=3$ which is consistent with the fact that $3$ is the lowest degree of any single polynomial defining $V$.

To test the singularity heuristic, we used the above samples and learned polynomials. For each value of $D$, we selected the best
performing value of $\eta$ and applied the singularity heuristic to the $3$ sample sets corresponding to that value of $D$ and $\eta$.
Then we measured the Wasserstein distance between $\operatorname{Sing} (\Omega)$ and the set of points $\operatorname{Sing}(\Omega')$ which passed the singularity heuristic. The average
values are given in Figure \ref{fig:7.2b}. The missing results are the cases where $\operatorname{Sing}(\Omega')$ is empty. In Figure \ref{fig:7.2c} we give the numbers of points that do pass the singularity heuristic. Once again, we see that the overall minimal distance is attained with $D=3$. 

To examine the effect of noise on the MAP model, we repeated the same process but added Gaussian noise with a standard deviation of $0.025$ to the original data $\Omega$. The result of this is illustrated in Figure \ref{fig:7.1c}. Samples from the learned polynomials were compared to the original noise-free data in order to test the robustness towards noise. The results are given in Figures \ref{fig:7.2d}, \ref{fig:7.2e}, and \ref{fig:7.2f}.

The values $D=1,2,3$ show similar trends with the overall minimal distances still attained at $D=3$. However the values $D=4,5$ show significantly worse performance which is due to over-fitting to noise. This is to be expected since the models for $D=4,5$ contain many more additional parameters which lead to higher model complexity.

\begin{figure}
\centering
\includegraphics[scale=.25]{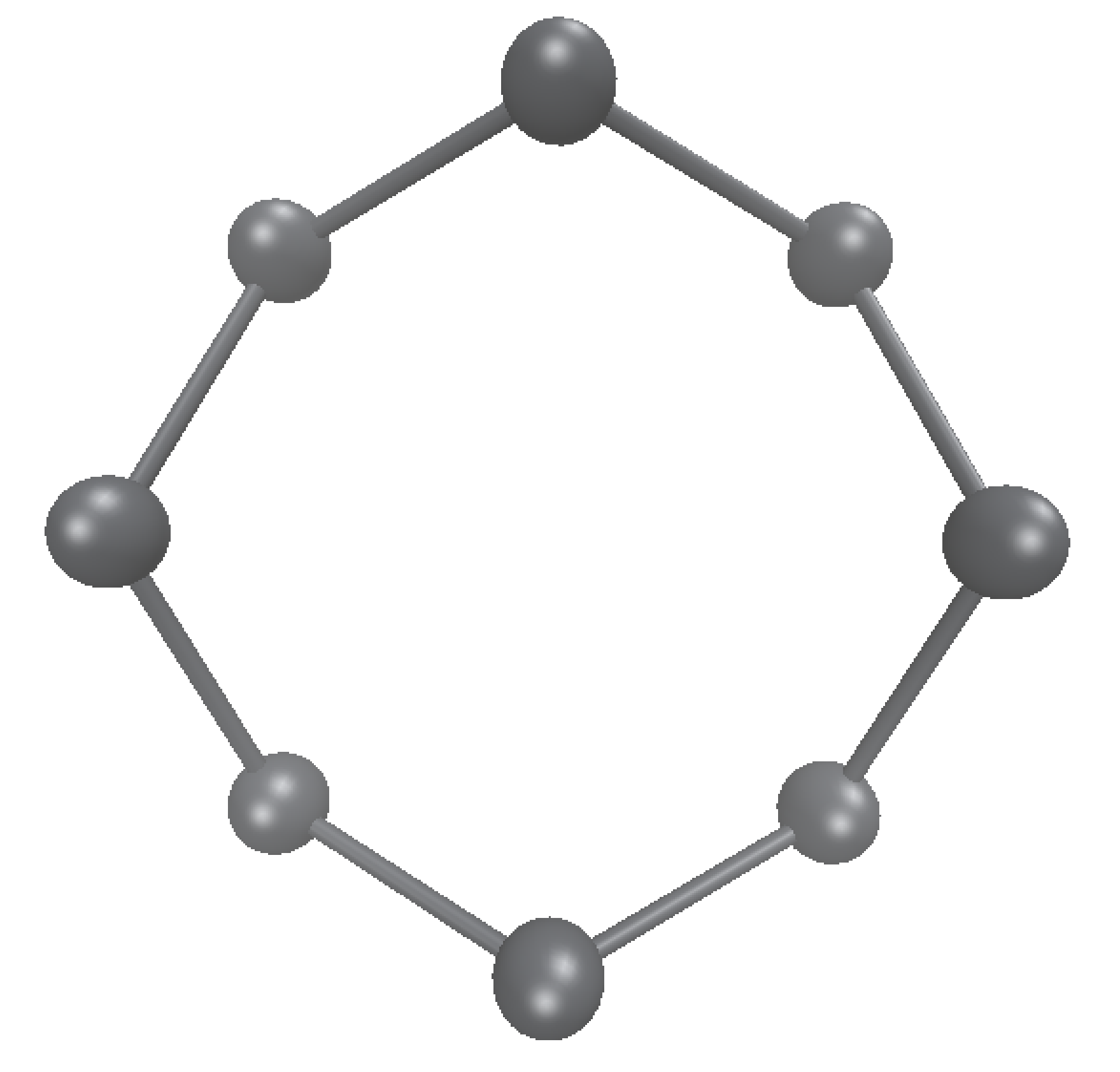}
\caption{Cyclooctane}\label{fig:7.3}
\end{figure}

\begin{example}
  Cyclooctane is a cyclic molecule with  chemical formula $(\text{CH}_{2})_{8}$ (Figure \ref{fig:7.3}).
%\color{red}
  The full energy landscape of cyclooctane in the sense of Born-Oppenheimer contains the positions of the carbon and hydrogen atoms and
  therefore has $72$ dimensions which is too high dimensional to be studied effectively by computational or numerical means. 
  The
  \textit{conformation space of cyclooctane} is a reduced energy landscape which consists of the chemically allowable
  positions for the eight carbon atoms differing from the lowest energy state %, usually at room temperature,
  only by rotation around single bonds. By geometric reasons, see e.g.\ \cite{Breiding2018}, 
%\color{black}
  the conformation space therefore is the variety cut out by a set of $16$ equations in 
  $\mathbb{R}[x_{1},y_{1},z_{1},\ldots ,x_{8},y_{8},z_{8}]$:
\begin{align*}
&(x_{i}-x_{i+1})^{2}+(y_{i}-y_{i+1})^{2}+(z_{i}-z_{i+1})^{2}-c = 0 \quad \text{and } \\
&(x_{i}-x_{i+2})^{2}+(y_{i}-y_{i+2})^{2}+(z_{i}-z_{i+2})^{2}-\frac{8}{3}c = 0 \quad \text{for all }i\in\mathbb{Z}_{8} \ .
\end{align*}
The constant $c$ is the bond length between two neigboring carbon atoms. By computational means, it has been determined to
have the value $c\thickapprox  2.21$ \cite{Breiding2018}. 

A dataset of $6040$ points sampled from this variety was produced by \cite{Martin2010}. 
The variety was further reduced 
to a subspace of $\R^5$. %euclidean space in dimension $5$. 
The reduction results in a compact $2$-dimensional singular surface 
which we will refer to as the \textit{reduced conformation space of cyclooctane}. 
The specific dimensionality reduction map consists of trigonometric functions so in particular is analytic; see \cite{Martin2010} for details. 
We are not aware of any proof that the reduced conformation space is also a variety, 
however, it is clear from the construction that the reduced conformation space is at least a compact subanalytic set 
by analyticity of the reduction map \cite{Bierstone1988}. Applying the reduction to the original $6040$ points, one obtains a set $\Omega$ of points lying on the reduced conformation space of cyclooctane.

We are particularly interested in the singular locus of the reduced conformation space. Based upon geometric considerations and the
sample tests discussed in \cite{Martin2010}, the reduced conformation space of cyclooctane is believed to be topologically equivalent to two M\"obius bands and a sphere glued together along two disjoint circles. Those two circles form the boundaries of the two M\"obius bands
and constitute the singular locus of the reduced conformation space. Chemically, the conformations represented by the two circles
are called \emph{peak (P)} and \emph{saddle (S)} and correspond to  maximal and minimal energy conformations, respectively, \cite{Martin2010}.
Using this information, we applied an adhoc method to extract the set of singular points from $\Omega$ resulting in a set $\operatorname{Sing}(\Omega$) of $233$ points shown in
Figures \ref{fig:7.4a}, \ref{fig:7.4b} and \ref{fig:7.4c} for different coordinate axes.

We applied the MAP model to the reduced conformation space for various values of $D$ but $D=4$ seemed to preform best. From the learned polynomial, we sampled a set $\Omega'$ of $40000$ points at threshold value $\eta = 10^{-7}$. We applied the singularity heuristic at a threshold value of $\varepsilon = 0.0003$ and obtained a set $\operatorname{Sing}(\Omega')$ of $235$ points. The set $\operatorname{Sing}(\Omega')$ is shown in
Figures \ref{fig:7.4d}, \ref{fig:7.4e} and \ref{fig:7.4f}.

The Wasserstein distance between $\operatorname{Sing}(\Omega$) and $\operatorname{Sing}(\Omega'$) was found to be $0.022$. We can see that despite the presence of noise coming from the sampling, the singularity heuristic succeeds to capture the overall geometry of the singular locus of the reduced conformation space of cyclooctane. This is further evidence that the reduced conformation space of cyclooctane is likely an algebraic variety and is likely defined by a polynomial of degree $4$.

\begin{figure}[H]
\centering
\begin{subfigure}{0.3\textwidth}\centering
\includegraphics[scale=0.25]{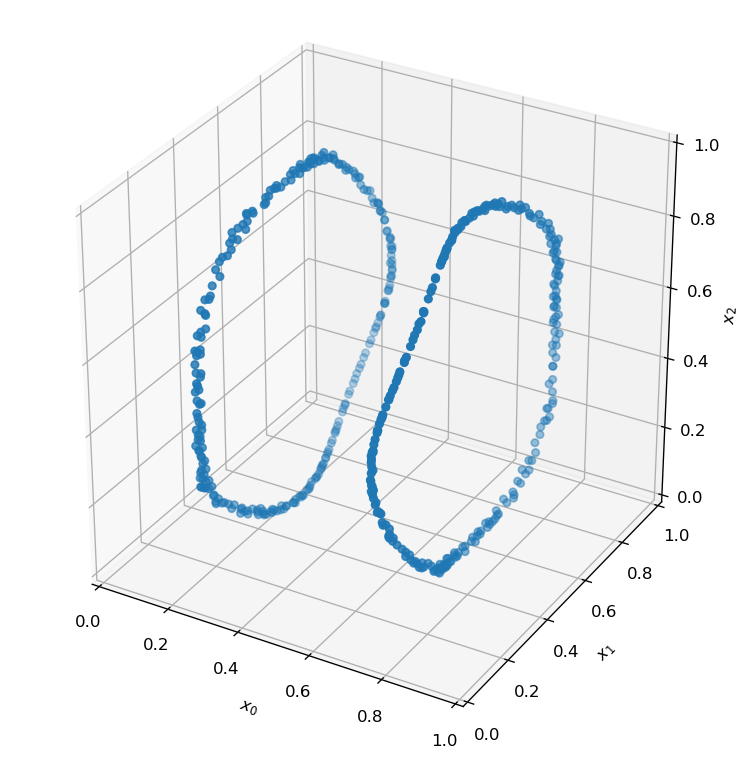}
\caption{}\label{fig:7.4a}
\end{subfigure}
\begin{subfigure}{.3\textwidth}\centering
\includegraphics[scale=0.25]{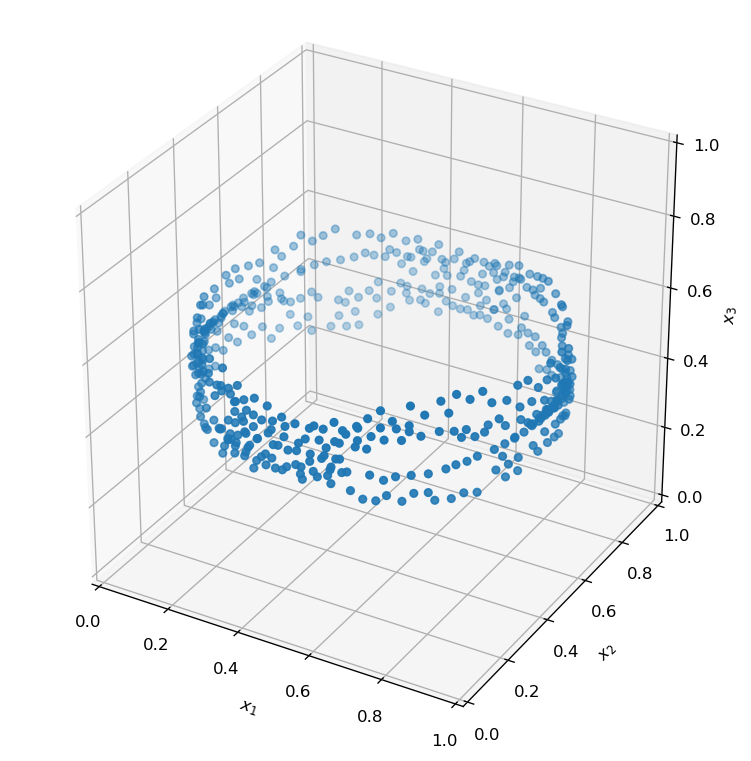}
\caption{}\label{fig:7.4b}
\end{subfigure}
\begin{subfigure}{.3\textwidth}\centering
\includegraphics[scale=0.25]{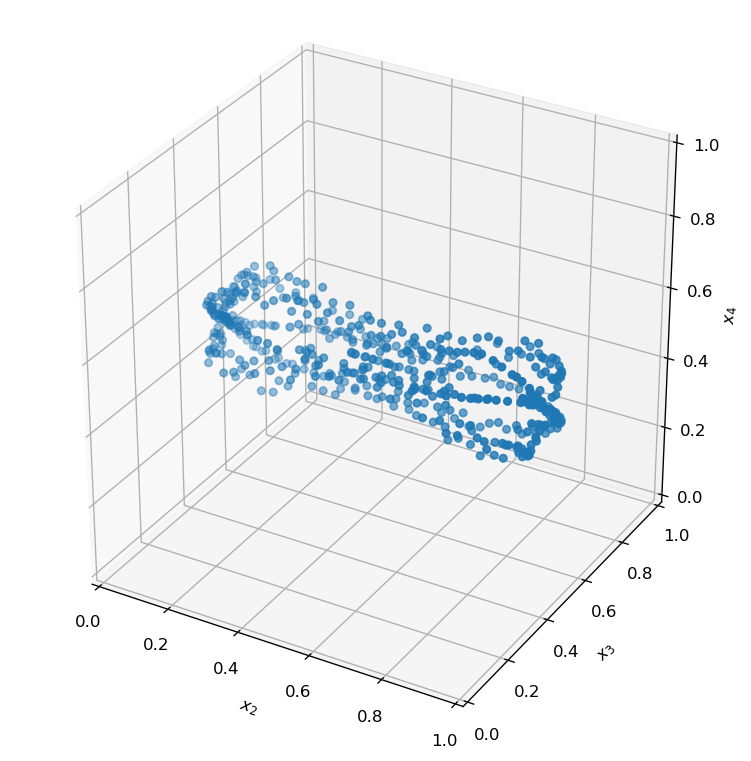}
\caption{}\label{fig:7.4c}
\end{subfigure}
\caption{Sampled singular locus of the conformational space of $(\text{CH}_2)_8$}
\end{figure}

\begin{figure}[H]
\centering
\begin{subfigure}{0.3\textwidth}\centering
\includegraphics[scale=0.25]{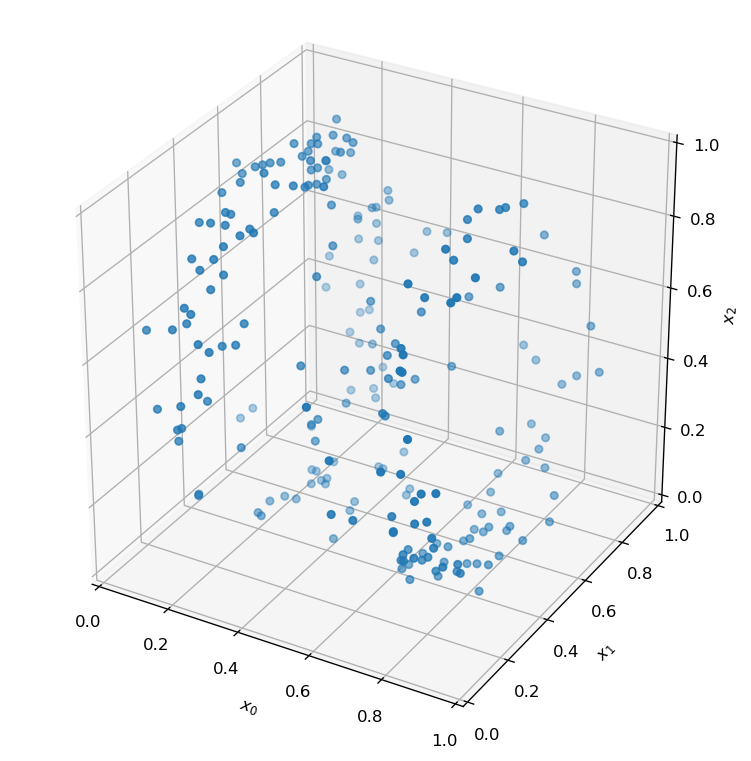}
\caption{}\label{fig:7.4d}
\end{subfigure}
\begin{subfigure}{.3\textwidth}\centering
\includegraphics[scale=0.25]{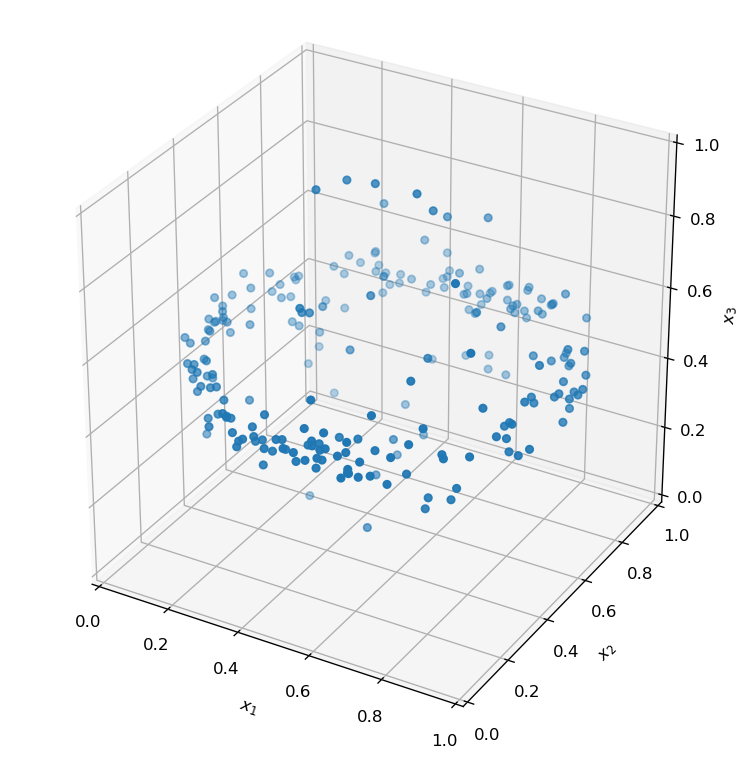}
\caption{}\label{fig:7.4e}
\end{subfigure}
\begin{subfigure}{.3\textwidth}\centering
\includegraphics[scale=0.25]{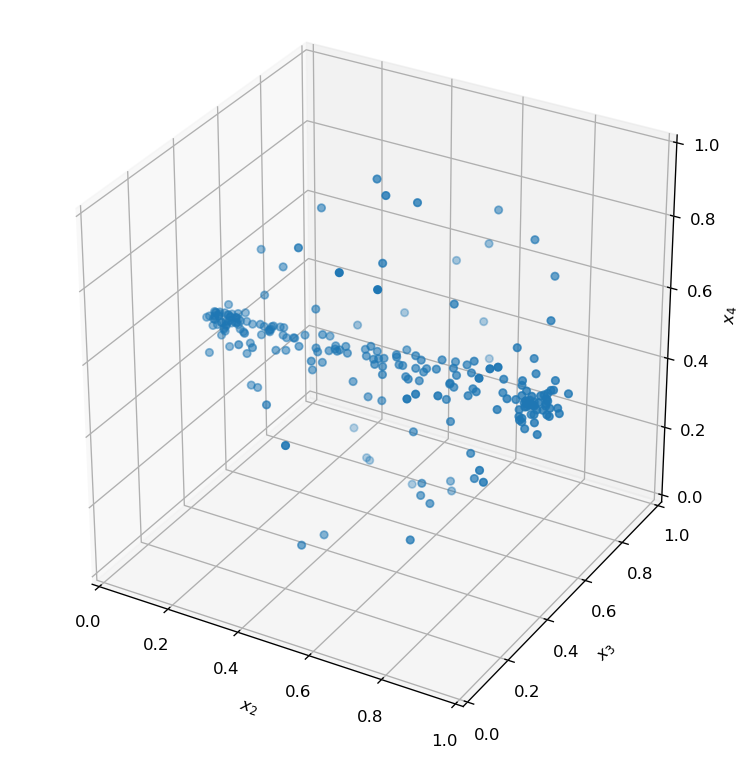}
\caption{}\label{fig:7.4f}
\end{subfigure}
\caption{Learned singular locus of the conformational space of $(\text{CH}_2)_8$}
\end{figure}

\end{example}

\begin{remark}
  The conformation space of cyclooctane has also been analyzed by Breiding et al.\ in \cite{Breiding2018},
  but the approach taken there is  different. In \cite{Breiding2018}, the authors do not consider the reduced conformation
  space but instead apply a mix of algebraic and topological data analysis methods to the unreduced conformation space.
  Based upon Martin et al.'s dataset \cite{Martin2010} sampled from the unreduced conformation space, 
  they determine in particular the dimension of the conformation space and compute the persistent homology of the dataset.   
  This approach returns the dimension of the conformation space to be $2$ and Betti numbers 
  indicating that the conformation space indeed should have a singular locus
  homeomorphic to the disjoint union of two circles.
  Note, however, that this method does not directly provide the particular location or coordinates of the singular points. Furthermore, Betti numbers do not always encode information as to whether the underlying space possesses singularities or not.
  %  Note, however, that the particular location or coordinates
  % of the singular points can not be made clear by that approach and that Betti numbers in general do not 
  % provide the information whether the underlying space possesses singularities or not.
  Finally, we remark that the method of learning varieties presented
  in \cite{Breiding2018} has been applied there to the unreduced conformation space of cyclooctane, only.
\end{remark}

\begin{remark}
%\color{red}
Let us briefly explain what has been achieved by considering the conformational space of
cyclooctane as an algebraic variety and comment on the potential broader use of singularity theory in chemistry.
According to the theoretical results by Hendrickson \cite{HendricksonCycooctane} and
Anet-Krane \cite{AnetKraneCycooctane}, cyclooctane has $10$ conformations:
chair, crown, boat, boat-chair, twist-boat-chair, boat-boat, twist-boat, twist-chair, chair-chair, and
twist-chair-chair. The existence of all these conformations has been verified experimentally.
The peak  and the saddle conformations described above by the singular locus of the reduced conformational
space were detected much later by the ground-breaking work \cite{Martin2010}, where sophisticated
methods from topological data analysis were used to study energy landscapes and conformational spaces.
We expect that to be  a more general phenomenon. Singular parts of (reduced) energy landscapes or
other chemically relevant surfaces such as the electron density or the electron localization function
are very small and difficult to find. But they might play an important role to theoretically
detect and describe new conformations or possibly even new reactions.  Further example studies from chemistry
are therefore intended. 
%\color{black}
\end{remark}

\section{Related Work}
\label{sec:related-work}
The MAP and intersected MAP model we developed in Section \ref{sec:learning-variety} are inspired
by the learning method given in \cite{Breiding2018}. 
The method presented in \cite{Breiding2018} aims to find the kernel of the 
Vandermonde matrix $U$ in a numerically stable manner. However, this is presented as a purely numerical method without a given statistical basis. 
As such, the authors focus on the case where $\Omega$ is sampled from a variety without any anomalies present.

In the case when there is no sample noise, the method of \cite{Breiding2018} coincides with the MAP model from
Section  \ref{sec:learning-variety} since in this case $\lambda=0$ and the MAP solutions would be given by $E_{\lambda}=\ker(U^{T}U)=\ker(U)$. 
In the case when there is noise, the authors use a noise tolerance parameter $\tau>0$. 
In one instance, they compute $\ker(U)$ by finding the singular vectors of $U$ whose singular values are 
$\leq \tau$. This is similar to the QCQP from Section \ref{sec:learning-variety} since the eigenvectors of
$U^{T}U$ are the right-singular vectors of $U$ and the eigenvalues of $U^{T}U$ are the squares of the 
singular values of $U$. The difference is that we only accept the smallest eigenvalue (or singular value). This can have a significant effect if the noise levels are high or the underlying space is only approximately a variety, as shown in 
Example \ref{Comparison} below. Intuitively, this effect appears when we have two varieties  $V_{1}$ and $V_{2}$ which fit the data $\Omega$ well but not exactly. In this case, the intersection  $V_{1}\cap V_{2}$ might not be close to the data set  $\Omega$ at all as shown in the following example.

%do not contain the set $\Omega$ but are both close to it,
%yet $V_{1}\cap V_{2}$ is not close to $\Omega$.

\begin{example}
\label{Comparison}
%\color{blue}
  Consider the noisy data $\Omega$ in Figure \ref{fig:8.1a} sampled from the plane $y-z=0$ with noise applied across the $z$ axis.
  Setting $D=1$, we find that the planes defined by the polynomials
  $\ell_{1}=0.73y-0.68z -0.03 $ (\ref{fig:8.1b}) and $\ell_{2}=0.17x+0.76y-0.61z-0.17$ (\ref{fig:8.1c}) are both in 
  $\bigoplus\limits_{\lambda < 0.003}E_{\lambda}$. Moreover, both planes are close to the original data. 
  However, the intersection $Z(\ell_{1})\cap Z(\ell_{2})$ is given by
  the noisy line in \ref{fig:8.1d} which has the wrong dimension.
% \color{black}  
\begin{figure}[H]
\centering
\begin{subfigure}{0.45\textwidth}
\includegraphics[scale=0.45]{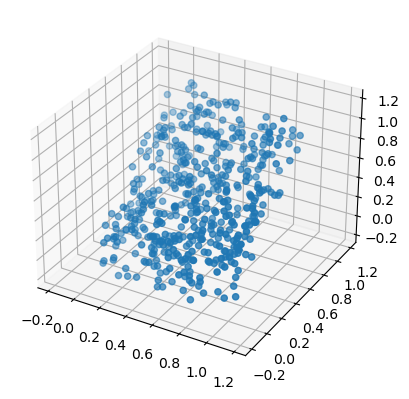}
\caption{}\label{fig:8.1a}
\end{subfigure}
\begin{subfigure}{.45\textwidth}
\includegraphics[scale=0.45]{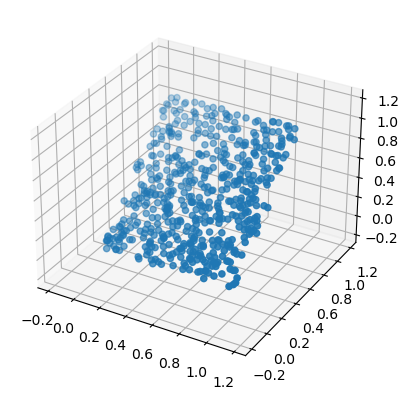}
\caption{}\label{fig:8.1b}
\end{subfigure}

\begin{subfigure}{0.45\textwidth}
\includegraphics[scale=0.45]{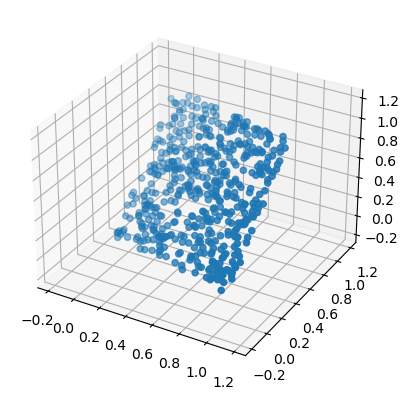}
\caption{}\label{fig:8.1c}
\end{subfigure}
\begin{subfigure}{.45\textwidth}
\includegraphics[scale=0.45]{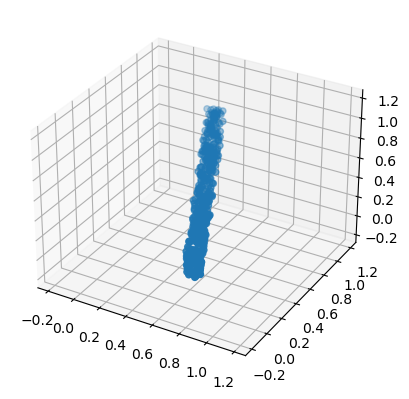}
\caption{}\label{fig:8.1d}
\end{subfigure}
\caption{}
\end{figure}
\end{example}

A related approach for singularity detection is taken in \cite{Stolz19664}. This approach involves local cohomology to test how well different regions of the data set can be approximated by Euclidean space. In contrast 
to algebraic varieties, the authors of \cite{Stolz19664} assume that the underlying geometry is a stratified space. However, their method does not rely on explicitly learning the underlying stratified space, and instead it uses persistent cohomology \cite{Carlsson2009} to detect the singular regions. This comes at a computational advantage, however, it also captures less information about the space itself. Furthermore, this method only applies to singularities where the space fails to be a topological manifold, and not the singularities where smoothness fails.

%%%%%%%%%%%%%%%%%%%%%%%%%%%%%%%%%%%%%%%%%%%%%%%%%%%%%%
%          AI TOOLS, USE AND LOCATION
%%%%%%%%%%%%%%%%%%%%%%%%%%%%%%%%%%%%%%%%%%%%%%%%%%%%%%
%We follow COPE's guidelines and policies regarding the use of Artificial Intelligence (AI) tools. COPE Policy on AI tools can be found at https://publicationethics.org/cope-position-statements/ai-author.

%Authors using AI tools in the writing of a manuscript, production of images or graphical elements of the paper, or in the collection and analysis of data, must be transparent in disclosing in this section how the AI tool was used and which tool was used. Authors are fully responsible for the content of their manuscript, even those parts produced by an AI tool, and are thus liable for any breach of publication ethics. - COPE

%Disclosure instructions

%If there is nothing to disclose, there is no need to add a declaration, otherwise please declare.

%\section*{Use of AI tools declaration}
%The author(s) declare(s) they have used Artificial Intelligence (AI) tools in the creation of this article.
%AI tools used:
%How were the AI tools used? 
%Where in the article is the information located?

%The acknowledgments section should not be numbered.
\section*{Acknowledgments}
The authors greatfully acknowledge support by the NSF under award OAC 1934725 for the project
\emph{DELTA: Descriptors of Energy Landscape by Topological Analysis}. M.J.\ Pflaum thanks David Jonas and David Walba for
crucial explanations  on the conformations of cylooctane. The authors also thank Henry Adams, Howy Jordan and Trevor Sesnic for helpful discussions.

%%%%%%%%%%%%%%%%%%%%%%%%%%%%%%%%%%%%%%%%%%%%%%%%%%%%%%
%          7. REFERENCES SECTION
%%%%%%%%%%%%%%%%%%%%%%%%%%%%%%%%%%%%%%%%%%%%%%%%%%%%%%

%       READ THIS SECTION CAREFULLY

% Each of the references below MUST be cited in your article above. Do not include references that are not cited in your article.

% Follow the examples below carefully. We strongly suggest that you copy and paste your reference information directly into our examples.

% List all references in alphabetical order according to the first author's last name.

% Verify each URL works correctly and can be accessed properly. Your URL links should be to reputable websites. The command line for a website link begins with: \url{ }

% Do not add MR or DOI numbers to your references. AIMS production staff will add this information.

% Using BibTex is not recommended but can be handled.

\bibliography{mybib}

\providecommand{\href}[2]{#2}
\providecommand{\arxiv}[1]{\href{http://arxiv.org/abs/#1}{arXiv:#1}}
\providecommand{\url}[1]{\texttt{#1}}
\providecommand{\urlprefix}{URL }
\begin{thebibliography}{10}

\bibitem{AnetKraneCycooctane}
\newblock F.~Anet and J.~Krane,
\newblock Strain energy calculation of conformations and conformational changes
  in cyclooctane,
\newblock \emph{Tetrahedron Letter}, \textbf{8} (1973), 5029--5052.

\bibitem{barber}
\newblock D.~Barber,
\newblock \emph{Bayesian Reasoning and Machine Learning},
\newblock Cambridge University Press, USA, 2012.

\bibitem{Becker1993}
\newblock E.~Becker and R.~Neuhaus,
\newblock Computation of real radicals of polynomial ideals,
\newblock in \emph{Computational Algebraic Geometry}, vol. 109 of Progr.\
  Math.,
\newblock Birkh\"{a}user Boston, 1993,
\newblock 1--20,
\newblock \urlprefix\url{https://doi.org/10.1007/978-1-4612-2752-6_1}.

\bibitem{Bierstone1988}
\newblock E.~Bierstone and P.~D. Milman,
\newblock Semianalytic and subanalytic sets,
\newblock \emph{Publications math{\'{e}}matiques de l'IH{\'{E}}S}, \textbf{67}
  (1988), 5--42,
\newblock \urlprefix\url{https://doi.org/10.1007/bf02699126}.

\bibitem{BocCosRoyRAG}
\newblock J.~Bochnak, M.~Coste and M.-F. Roy,
\newblock \emph{Real algebraic geometry}, vol.~36 of Ergebnisse der Mathematik
  und ihrer Grenzgebiete (3) [Results in Mathematics and Related Areas (3)],
\newblock Springer-Verlag, Berlin, 1998,
\newblock Translated from the 1987 French original, Revised by the authors.

\bibitem{Convex}
\newblock S.~Boyd and L.~Vandenberghe,
\newblock \emph{Convex Optimization},
\newblock Cambridge University Press, 2004.

\bibitem{Breiding2018}
\newblock P.~Breiding, S.~Kali{\v{s}}nik, B.~Sturmfels and M.~Weinstein,
\newblock Learning algebraic varieties from samples,
\newblock \emph{Revista Matem{\'{a}}tica Complutense}, \textbf{31} (2018),
  545--593,
\newblock \urlprefix\url{https://doi.org/10.1007/s13163-018-0273-6}.

\bibitem{10.1007/978-3-319-96418-8_54}
\newblock P.~Breiding and S.~Timme,
\newblock Homotopycontinuation.jl: A package for homotopy continuation in
  julia,
\newblock in \emph{Mathematical Software -- ICMS 2018} (eds. J.~H. Davenport,
  M.~Kauers, G.~Labahn and J.~Urban),
\newblock Springer International Publishing, Cham, 2018,
\newblock 458--465.

\bibitem{Carlsson2009}
\newblock G.~Carlsson,
\newblock Topology and data,
\newblock \emph{Bulletin of the American Mathematical Society}, \textbf{46}
  (2009), 255--308.

\bibitem{ColdingMinicozzi}
\newblock T.~H. Colding and W.~P. Minicozzi~II,
\newblock {\L}ojasiewicz inequalities and applications,
\newblock in \emph{Surveys in {D}ifferential {G}eometry 2014. {R}egularity and
  evolution of nonlinear equations}, vol.~19 of Surv.\ Differ.\ Geom.,
\newblock Int. Press, Somerville, MA, 2015,
\newblock 63--82,
\newblock \url{https://doi.org/10.4310/SDG.2014.v19.n1.a3}.

\bibitem{SINGULAR}
\newblock W.~Decker, G.-M. Greuel, G.~Pfister and H.~Sch\"onemann,
\newblock {\sc Singular} {4-2-0} - {A} computer algebra system for polynomial
  computations,
\newblock \url{http://www.singular.uni-kl.de}, 2020.

\bibitem{8999343}
\newblock E.~{Dufresne}, P.~{Edwards}, H.~{Harrington} and J.~{Hauenstein},
\newblock Sampling real algebraic varieties for topological data analysis,
\newblock in \emph{2019 18th IEEE International Conference On Machine Learning
  And Applications (ICMLA)}, 2019,
\newblock 1531--1536.

\bibitem{doi:10.1021/jp962746i}
\newblock M.~I. Dykman, V.~N. Smelyanskiy, R.~S. Maier and M.~Silverstein,
\newblock Singular features of large fluctuations in oscillating chemical
  systems,
\newblock \emph{The Journal of Physical Chemistry}, \textbf{100} (1996),
  19197--19209,
\newblock \urlprefix\url{https://doi.org/10.1021/jp962746i}.

\bibitem{Fefferman2016}
\newblock C.~Fefferman, S.~Mitter and H.~Narayanan,
\newblock Testing the manifold hypothesis,
\newblock \emph{Journal of the American Mathematical Society}, \textbf{29}
  (2016), 983--1049,
\newblock \urlprefix\url{https://doi.org/10.1090/jams/852}.

\bibitem{Gfvert2020}
\newblock O.~G\"{a}fvert,
\newblock Computational complexity of learning algebraic varieties,
\newblock \emph{Advances in Applied Mathematics}, \textbf{121} (2020), 102100,
\newblock \urlprefix\url{https://doi.org/10.1016/j.aam.2020.102100}.

\bibitem{Gander1994}
\newblock W.~Gander, G.~H. Golub and R.~Strebel,
\newblock Least-squares fitting of circles and ellipses,
\newblock \emph{{BIT}}, \textbf{34} (1994), 558--578,
\newblock \urlprefix\url{https://doi.org/10.1007/bf01934268}.

\bibitem{10.5555/1557288}
\newblock G.-M. Greuel and G.~Pfister,
\newblock \emph{A Singular Introduction to Commutative Algebra},
\newblock 2nd edition,
\newblock Springer Publishing Company, Incorporated, 2007.

\bibitem{HarAG}
\newblock R.~Hartshorne,
\newblock \emph{Algebraic Geometry},
\newblock Graduate Texts in Mathematics, Springer New York, 2013.

\bibitem{HendricksonCycooctane}
\newblock J.~B. Hendrickson,
\newblock Molecular geometry.\ {V}.\ {E}valuation of functions and
  conformations of medium rings,
\newblock \emph{J.\ Am.\ Chem.\ Soc.}, \textbf{89} (1967), 7047–7061.

\bibitem{HUYNH1986196}
\newblock D.~T. Huynh,
\newblock A superexponential lower bound for {G}röbner bases and
  {C}hurch-{R}osser commutative thue systems,
\newblock \emph{Information and Control}, \textbf{68} (1986), 196--206,
\newblock
  \urlprefix\url{https://www.sciencedirect.com/science/article/pii/S0019995886800353}.

\bibitem{Lang:Algebra}
\newblock S.~Lang,
\newblock \emph{Algebra}, vol. 211 of Graduate Texts in Mathematics,
\newblock 3rd edition,
\newblock Springer-Verlag, New York, 2002,
\newblock \urlprefix\url{https://doi.org/10.1007/978-1-4613-0041-0}.

\bibitem{Leeuw2013HistoryON}
\newblock J.~Leeuw,
\newblock History of nonlinear principal component analysis, 
\newblock UCLA: Department of Statistics, UCLA (2013), \url{https://escholarship.org/uc/item/1vp9f9kz}.

\bibitem{LojESA}
\newblock S.~{\L}ojasiewicz,
\newblock Ensemble semi-analytique, 1965,
\newblock Mimeographi{\'e}, {I}nstitute des {H}autes {\'E}tudes {S}cientifique,
  {B}ures-sur-{Y}vette, {F}rance.

\bibitem{Martin2010}
\newblock S.~Martin, A.~Thompson, E.~A. Coutsias and J.-P. Watson,
\newblock Topology of cyclo-octane energy landscape,
\newblock \emph{The Journal of Chemical Physics}, \textbf{132} (2010), 234115,
\newblock \urlprefix\url{https://doi.org/10.1063/1.3445267}.

\bibitem{mcinnes2020umap}
\newblock L.~McInnes, J.~Healy and J.~Melville,
\newblock UMAP: Uniform manifold approximation and projection for dimension
  reduction, 
\newblock \href{https://arxiv.org/abs/1802.03426}{arXiv:1802.03426} (2020).

\bibitem{AhmadiMoughari2020ADRMLAD}
\newblock F.~A. Moughari and C.~Eslahchi,
\newblock ADRML: Anticancer drug response prediction using manifold learning,
\newblock \emph{Scientific Reports}, \textbf{10}.

\bibitem{nash1952}
\newblock J.~Nash,
\newblock Real algebraic manifolds,
\newblock \emph{Ann.\ of Math.\ (2)}, \textbf{56} (1952), 405--421.

\bibitem{Neuhaus1998}
\newblock R.~Neuhaus,
\newblock Computation of real radicals of polynomial ideals. {II},
\newblock \emph{J. Pure Appl. Algebra}, \textbf{124} (1998), 261--280,
\newblock \urlprefix\url{https://doi.org/10.1016/S0022-4049(96)00103-X}.

\bibitem{10.5555/1145132}
\newblock J.~Park,
\newblock \emph{Manifold Learning in Computer Vision},
\newblock PhD thesis, USA, 2005,
\newblock AAI3193223.

\bibitem{PflAGSSS}
\newblock M.~J. Pflaum,
\newblock \emph{Analytic and geometric study of stratified spaces}, vol. 1768
  of Lecture Notes in Mathematics,
\newblock Springer-Verlag, Berlin, 2001.

\bibitem{reiman1996singular}
\newblock A.~Reiman,
\newblock Singular surfaces in the open field line region of a diverted
  tokamak,
\newblock \emph{Physics of Plasmas}, \textbf{3} (1996), 906--913.

\bibitem{ruschendorf1985wasserstein}
\newblock L.~R{\"u}schendorf,
\newblock The {W}asserstein distance and approximation theorems,
\newblock \emph{Probability Theory and Related Fields}, \textbf{70} (1985),
  117--129.

\bibitem{Solern1991}
\newblock P.~Solernó,
\newblock Effective Łojasiewicz inequalities in semialgebraic geometry,
\newblock \emph{Applicable Algebra in Engineering, Communication and
  Computing}, \textbf{2} (1991), 1–14,
\newblock \urlprefix\url{http://dx.doi.org/10.1007/BF01810850}.

\bibitem{realrad}
\newblock S.~J. Spang,
\newblock \emph{On the computation of the real radical.},
\newblock PhD thesis, Technische Universit{\"a}t Kaiserslautern, 2007.

\bibitem{Stolz19664}
\newblock B.~J. Stolz, J.~Tanner, H.~A. Harrington and V.~Nanda,
\newblock Geometric anomaly detection in data,
\newblock \emph{Proceedings of the National Academy of Sciences}, \textbf{117}
  (2020), 19664--19669,
\newblock \urlprefix\url{https://www.pnas.org/content/117/33/19664}.

\bibitem{Tenenbaum2000AGG}
\newblock J.~Tenenbaum, V.~D. Silva and J.~Langford,
\newblock A global geometric framework for nonlinear dimensionality reduction.,
\newblock \emph{Science}, \textbf{290 5500} (2000), 2319--23.

\bibitem{Whitney65}
\newblock H.~Whitney,
\newblock Local properties of analytic varieties,
\newblock in \emph{Differential and {C}ombinatorial {T}opology ({A} {S}ymposium
  in {H}onor of {M}arston {M}orse)},
\newblock Princeton Univ. Press, Princeton, N.J., 1965,
\newblock 205--244.

\end{thebibliography}
\bibliographystyle{AIMS}

\end{document}